\documentclass[a4paper,10pt]{article}

\usepackage{amsmath,amsthm,amssymb,eufrak}

\usepackage[charter]{mathdesign}

\usepackage[margin=3cm]{geometry}
\usepackage{hyperref}

\newtheorem{theorem}{Theorem}[section]

\newtheorem{lemma}[theorem]{Lemma}
\newtheorem{corollary}[theorem]{Corollary}
\newtheorem{remark}[theorem]{Remark}

\newtheorem{example}[theorem]{Example}
\newtheorem{assumption}[theorem]{Assumption}

\begin{document}

\title{On the ergodicity of certain Markov chains in random environments}

\author{Bal\'azs Gerencs\'er\thanks{Alfr\'ed R\'enyi Institute of Mathematics and E\"otv\"os Lor\'and University, Budapest, Hungary}
\and Mikl\'os R\'asonyi\thanks{Alfr\'ed R\'enyi Institute of Mathematics, Budapest, Hungary}}

\date{\today}

\maketitle

\begin{abstract}
We study the ergodic behaviour of a discrete-time process $X$ which is a Markov chain 
in a stationary random environment. The laws of $X_t$ are shown to converge
to a limiting law in (weighted) total variation distance as $t\to\infty$.
Convergence speed is estimated and an ergodic theorem is established for functionals of $X$. 

Our hypotheses on $X$ combine the standard ``drift'' and ``small set'' conditions for geometrically ergodic Markov chains
with conditions on the growth rate of a certain ``maximal process'' of the random environment.
We are able to cover a wide range of models that have heretofore been
untractable. In particular, our results are pertinent to difference equations modulated by
a stationary (Gaussian) process. Such equations arise in applications such as 
discretized stochastic volatility models
of mathematical finance. 
\end{abstract}

\section{Introduction}\label{intro}

Markov chains in random environments (\emph{recursive chains} in the terminology of \cite{borovkov})
were systematically studied on countable state spaces in
e.g.\ \cite{c1,c2,o}. However, papers on the ergodic properties of such processes on a general state space
are scarce and require rather strong, Doeblin-type conditions, see \cite{kifer1,kifer,sep}. An exception
is \cite{stenflo}, where the system dynamics is assumed to be contracting instead. This is also
rather restrictive an assumption and 
only weak convergence of the laws can be established.

In this paper we deal with Markov chains in random environments that satisfy refinements of the usual
hypotheses for the geometric ergodicity of Markov chains: minorization on ``small sets'', see Chapter 5
of \cite{mt}, and Foster--Lyapunov type ``drift'' conditions, see Chapter 15 of \cite{mt}. 

Assuming that a suitably defined maximal process of the random environment satisfies a tail estimate,
we manage to establish stochastic stability. We use certain ideas of M. Hairer and J. Mattingly (see \cite{hm}) to obtain convergence 
to a limiting distribution in total variation norm with 
estimates on the convergence rate, see Section \ref{mar} for the statements of our results. 
We also present a method to prove ergodic theorems, exploiting ideas of 
\cite{bm,bw,gmmtv,msg}. An important technical ingredient is the notion of $L$-mixing, see Section \ref{lm}. 

We present difference equations modulated by Gaussian processes in Section \ref{dif}, to which our results apply. 
These can be regarded as discretizations
of diffusions in random environments which arise, for instance, in stochastic volatility models of mathematical finance, 
see \cite{cr} and \cite{gjr}. These examples allow to demonstrate the power of our approach, hopefully to be 
followed by other applications as our main purpose is to exhibit a versatile \emph{method} to be used in the future.
Proofs appear in Sections \ref{prof}, \ref{prof1} and \ref{mehi}.

\section{Main results}\label{mar}

Let $\mathcal{Y}$ be a Polish space with its Borel sigma-field $\mathfrak{A}$ and let $Y_t$, $t\in\mathbb{Z}$
be a (strongly) stationary $\mathcal{Y}$-valued process on some probability space
$(\Omega,\mathcal{F},P)$. A generic element of $\Omega$ will be denoted by $\omega$.

Expectation of a real-valued random variable $X$ with respect to $P$ will be
denoted by $E[X]$ in the sequel. For $1\leq p<\infty$ we write $L^p$ to denote the Banach space of 
(a.s.\ equivalence classes of)
$\mathbb{R}$-valued random variables with $E[|X|^p]<\infty$, equipped with the usual norm. 

We fix another Polish space $\mathcal{X}$ with its Borel sigma-field $\mathfrak{B}$ and denote by $\mathcal{P}(\mathcal{X})$ the set of 
probability measures on $\mathfrak{B}$.
Let $Q:\mathcal{Y}\times\mathcal{X}\times \mathfrak{B}\to [0,1]$ be a family of probabilistic kernels parametrized by 
$y\in\mathcal{Y}$, i.e. for all $A\in \mathfrak{B}$,
$Q(\cdot,\cdot,A)$ is $\mathfrak{A}\otimes\mathfrak{B}$-measurable 
and for all $y\in\mathcal{Y}$, $x\in\mathcal{X}$, $A\to Q(y,x,A)$
is a probability on $\mathfrak{B}$. 

Let ${X}_t$, $t\in\mathbb{N}$ be a $\mathcal{X}$-valued stochastic
process such that 
\begin{equation}\label{recu}
P({X}_{t+1}\in A\vert\mathcal{F}_t)=Q(Y_{t},X_t,A)\ P\mbox{-a.s.},\ t\geq 0, 
\end{equation}
where the filtration is defined by
$$
\mathcal{F}_t:=\sigma(Y_j,\ j\in\mathbb{Z};\ X_j,\ 0\leq j\leq t),\ t\geq 0.
$$

\begin{remark} {\rm Obviously, the law of 
${X}_t$, 
$t\in\mathbb{N}$ (and also its joint law with 
$Y_t$, $t\in\mathbb{Z}$) are
uniquely determined by \eqref{recu}.
For every given $Q$,
there exists a process ${X}$ satisfying \eqref{recu} (after possibly enlarging the probability space).  
See e.g.\ page 228 of \cite{bwbook} for a similar construction. We will establish a more
precise result in Lemma \ref{t} below, under additional assumptions.}
\end{remark}

The process $Y$ will represent
the random environment whose state $Y_t$ at time $t$ determines the transition law 
$Q(Y_t,\cdot,\cdot)$ of the process $X$ at the given instant $t$.
Our purpose is to study the ergodic properties of $X$. 

We will now introduce a number of assumptions of various kinds that will figure in the
statements of the main results: Theorems \ref{limit}, \ref{limit2}, \ref{lln} and \ref{llnn} below.

The following assumption closely resembles the well-known drift conditions
for geometrically ergodic Markov chains, see e.g.\ Chapter 15 of \cite{mt}. In our
case, however, they are relaxed by also having dependence on the state of the random environment.

\begin{assumption}\label{lyapunov} (Drift condition) Let 
$V:\mathcal{X}\to \mathbb{R}_{+}$ be a measurable function. Let $A_n\in\mathfrak{A}$, $n\in\mathbb{N}$ be a non-decreasing sequence of
subsets such that $A_0\neq\emptyset$ and $\mathcal{Y}=\cup_{n\in\mathbb{N}}A_n$.
Define the $\mathbb{N}$-valued function 
\[
\Vert y\Vert:=\min\{n:\, y\in A_n\},\ y\in\mathcal{Y}.
\]
We assume that there is a non-increasing function $\lambda:\mathbb{N}\to (0,1]$ and a non-decreasing function $K:\mathbb{N}\to (0,\infty)$
such that, for all $x\in\mathcal{X}$ and $y\in\mathcal{Y}$,
\begin{equation}\label{lyapi}
\int_{\mathcal{X}} V(z)\, Q(y,x,dz)\leq (1-\lambda(\Vert y\Vert))V(x)+ K(\Vert y\Vert). 
\end{equation}
Furthermore, we may and will assume $\lambda(\cdot)\leq 1/3$ and
$K(\cdot)\geq 1$. 
\end{assumption}

We provide some intuition about Assumption \ref{lyapunov}: we expect that the stochastic process $X$ behaves in 
an increasingly arbitrary way as the random environment $Y$ becomes more and more ``extreme'' (i.e.\ $\Vert Y\Vert$ grows) so
the drift condition \eqref{lyapi} becomes less and less stringent (i.e.\ $\lambda(\Vert Y\Vert)$ decreases).
%on the increasing subsets $A_n$ as $n$ grows. 

\begin{example}\label{ppp}
{\rm A typical case is where $\mathcal{Y}$ is a subset of a Banach space $\mathbb{B}$ with norm 
$\Vert\cdot\Vert_{\mathbb{B}}$; $\mathfrak{A}$ its Borel field; $A_n:=\{y\in\mathcal{Y}:\, 
\Vert y\Vert_{\mathbb{B}}\leq n\}$, $n\in\mathbb{N}$. In this setting 
$$
\Vert y\Vert=\left\lceil \Vert y\Vert_{\mathbb{B}} \right\rceil,
$$ 
where $\lceil\cdot\rceil$ stands for the ceiling function.
In the examples of the present paper we will always have $\mathbb{B}=\mathbb{R}^d$ with some $d\geq 1$ and 
$|\cdot|=\Vert\cdot\Vert_{\mathbb{B}}$ will denote the respective Euclidean norm.}

%Another standard choice would be $\mathcal{Y}:=\mathbb{N}$; $\mathfrak{A}$ is the power set of $\mathcal{Y}$;
%$A_n:=\{i\in\mathbb{N}:\ i\leq n\}$. In this case $\Vert y\Vert=y$, $y\in\mathbb{N}$.

%One more possibility could be $\mathcal{Y}:=(0,\infty)$ with its Borel sets $\mathfrak{A}$ 
%and with $A_n:=[1/(n+1),\infty)$, $n\in\mathbb{N}$.}
\end{example}

\begin{remark}{\rm It would be desirable
to relax Assumption \ref{lyapunov} allowing $\lambda$ to vary in $(-\infty,1)$ as long as ``in the average'' it is contractive (there are multiple 
options for the precise formulation of such a property). This is, however, out of scope for the current work.} 
\end{remark}

The next assumption stipulates the existence of a whole family of suitable ``small sets'' $C(R(n))$ 
that fit well the sets $A_n$ appearing in Assumption \ref{lyapunov}. 

\begin{assumption}\label{small} (Minorization condition) For $R\geq 0$, set 
$C(R):=\{x\in\mathcal{X}:\ V(x)\leq R\}$. 
Let $\lambda(\cdot)$, $K(\cdot)$ be as in Assumption \ref{lyapunov}.
Define $R(n):=4K(n)/\lambda(n)$.
There is a non-increasing function $\alpha:\mathbb{N}\to (0,1]$ and for each $n\in\mathbb{N}$, there exists a 
probability measure $\nu_n$ on $\mathfrak{B}$ 
such that, for all $y\in\mathcal{Y}$, $x\in C(R(\Vert y\Vert))$
and $A\in\mathfrak{B}$,
\begin{equation}\label{mey}
Q(y,x,A)\geq \alpha(\Vert y\Vert)\nu_{\Vert y\Vert}(A).
\end{equation}
We may and will assume $\alpha(\cdot)\leq 1/3$.
\end{assumption}

In other words,
depending on the ``size'' $\Vert y\Vert$ of state $y$ of the random environment, we work on the set $C(4K(\Vert y\Vert)/\lambda(\Vert y\Vert))$ on which we are able to benefit 
from a ``coupling effect'' of 
strength $\alpha(\Vert y\Vert)$.

For a fixed $V$ as in Assumption \ref{lyapunov}, let us define a family of metrics on
\[
\mathcal{P}_V(\mathcal{X}):=\left\{\mu\in\mathcal{P}(\mathcal{X}):\, \int_{\mathcal{X}} V(x)\, \mu(dx)
<\infty\right\}
\]
by 
\[
\rho_{\beta}(\nu_1,\nu_2):=\int_{\mathcal{X}} [1+\beta V(x)]\vert \nu_1-\nu_2\vert(dx),\ 
\nu_1,\nu_2\in\mathcal{P}_V(\mathcal{X}),
\]
for each $0\leq\beta\leq 1$. Here $\vert \nu_1-\nu_2\vert$ is the total variation of
the signed measure $\nu_1-\nu_2$.
Note that $\rho_0$ is just the total variation distance (and it can be defined for all 
$\nu_1,\nu_2\in\mathcal{P}(\mathcal{X})$)  while 
$\rho_1$ is the $(1+V)$-weighted total variation distance.

For a measurable $f:\mathcal{X}\to\mathbb{R}_{+}$ we define $\Phi(f)$ to be the set of measurable $\phi:\mathcal{X}\to \mathbb{R}$
such that $|\phi(z)|\leq C(1+f(z))$, $z\in\mathcal{X}$ holds for some constant $C=C(\phi)$. Hence $\Phi(1)$ denotes the set of
bounded, measurable functions on $\mathcal{X}$.

Let $L:\mathcal{X}\times\mathfrak{B}\to [0,1]$ be a probabilistic kernel. 
For each $\mu\in\mathcal{P}(\mathcal{X})$, we define the probability  
\begin{equation}\label{taho}
[L\mu](A):=\int_{\mathcal{X}} L(x,A)\, \mu(dx),\ A\in\mathfrak{B}.
\end{equation}
Consistently with these definitions, $Q(Y_n)\mu$ 
will refer
to the action of the kernel $Q(Y_n,\cdot,\cdot)$ on $\mu$.
Note, however, that $Q(Y_n)\mu$ is a \emph{random} probability measure.

For a bounded measurable function $\phi:\mathcal{X}\to\mathbb{R}$,
we set 
\[
[L\phi](x):=\int_{\mathcal{X}}\phi(z)\, L(x,dz),\ x\in\mathcal{X}.
\]
The latter definition makes sense for any non-negative measurable $\phi$, too.

Introduce the notation $\mu_t:=\mathrm{Law}(X_t)$, $t\in\mathbb{N}$.
The following assumption is an integrability condition about the initial values $X_0$ and $X_1$ of the process $X$.
\begin{assumption}\label{init} (Moment condition on the initial values)
\begin{equation*}
E[V^{2}(X_{0})+V^{2}(X_{1})]<\infty.
\end{equation*}
\end{assumption}

We now present a hypothesis controlling the maxima of $\Vert Y\Vert$ over finite time intervals (i.e. the ``degree of extremity''
of the random environment). 

\begin{assumption}\label{stary} (Condition on the maximal process of the
random environment)
There exist a non-decreasing function $g:\mathbb{N}\to\mathbb{N}$ and a non-increasing function $\ell:\mathbb{N}\to [0,1]$ such that
\begin{equation}\label{mata}
P\left(\max_{1\leq i\leq t}\Vert Y_i\Vert\geq g(t)\right)\leq \ell(t),\ t\geq 1.
\end{equation}
\end{assumption}

\begin{remark} {\rm It is clear that for a given process $Y$, several choices for the pair of 
functions $g,\ell$ are possible.
Each of these leads to different estimates and it depends on $Y$ and $X$ which choice is better, no general
rule can be determined a priori.}
\end{remark}

\begin{remark} {\rm For Gaussian processes $Y$ in $\mathcal{Y}=\mathbb{R}^d$, Assumption \ref{stary} holds, for instance,
with $g(t)\sim \sqrt{t}$ and $\ell(t)$ eponentially decreasing, see Section \ref{dif} for more details.}
\end{remark}

\begin{remark}\label{l+} {\rm One can derive estimates like \eqref{mata} also for rather general processes $Y$. For instance, 
let $Y_t$, $t\in\mathbb{Z}$ be $\mathbb{R}^d$-valued strongly stationary
such that $E|Y_0|^p<\infty$ for all $p\geq 1$. Then for each $q\geq 1$ set $p=2q$ and estimate 
\begin{eqnarray*}
& & E^{1/q}\left[\max_{1\leq i\leq t}|Y_i|^q\right]\leq E^{1/2q}\left[\max_{1\leq i\leq t}|Y_i|^{2q}\right]\\
&\leq& E^{1/2q}\left[\sum_{i=1}^t|Y_i|^{2q}\right]\leq  C(q) t^{\frac{1}{2q}},
\end{eqnarray*}
with constant $C(q)=E^{1/2q}[|Y_0|^{2q}]$. The Markov inequality implies that
\begin{equation}\label{trisk}
P\left(\max_{1\leq i\leq t}|Y_i|\geq t\right)\leq \frac{C^q(q) t^{1/2}}{t^q}\leq \frac{C^q(q)}{t^{q-1/2}}.
\end{equation}
Actually, for arbitrarily small $\chi>0$ and arbitrarily large $r\geq 1$, we can set $q=\frac{r}{\chi}+\frac{1}{2}$
in \eqref{trisk} and then Assumption \ref{stary} holds with 
$$
g(t)=\lceil t^{\chi}\rceil\mbox{ and }
\ell(t)=\frac{C^{q}(q)}{t^r},\ t\geq 1,
$$
i.e.\ for arbitrary polynomially growing $g(\cdot)$ and polynomially
decreasing $\ell(\cdot)$. 
This shows that our main results below have a 
wide spectrum of applicability 
well beyond the case of Gaussian $Y$, see also Example \ref{mbaker} below.}
\end{remark}

We now define a number of quantities that will appear in various convergence rate estimates below. For each $t\in\mathbb{N}$, set

\begin{eqnarray*}
r_1(t) &:=& \sum_{k=t}^{\infty} \frac{K(g(k))}{\alpha(g(k))}e^{-k\alpha(g(k))\lambda(g(k))/2},\\
r_2(t) &:=& \sum_{k=t}^{\infty} \frac{K(g(k+1))}{\alpha^2(g(k+1))\lambda(g(k+1))}\sqrt{\ell(k)},\\
r_3(t) &:=& \sum_{k=t}^{\infty} e^{-k\alpha(g(k))\lambda(g(k))/2},\\
r_4(t) &:=& \sum_{k=t}^{\infty} \ell(k),\\
\pi(t) &:=& \frac{|\ln(\lambda(g(t)))|}{\alpha(g(t))\lambda(g(t))}.
\end{eqnarray*}

Now comes the first main result of the present paper: assuming our conditions on drift, minorization, initial
values and control of the maxima, $\mu_t$ will tend to a limiting law as 
$t\to\infty$, provided that $r_1(0)$
and $r_2(0)$ are finite.

\begin{theorem}\label{limit} Let Assumptions \ref{lyapunov}, \ref{small} and \ref{stary}  
be in force. Assume 
\begin{equation}\label{psota1}
r_1(0)+r_2(0)<\infty.
\end{equation} 
Then there is a probability
$\mu_*$ on $\mathcal{X}$ such that, for any $X_{0}$ satisfying Assumption \ref{init},
$\mu_t\to \mu_*$ in $(1+V)$-weighted
total variation as $t\to\infty$.
More precisely,
\[
\rho_1(\mu_t,\mu_*)\leq C[r_1(t)+r_2(t)],\ t\in\mathbb{N},
\] 
for some constant $C>0$. The limit $\mu_{*}$ does not depend on $X_{0}$.
\end{theorem}

Theorem \ref{limit2} below is just a variant of Theorem \ref{limit}: relaxing the assumptions it
provides convergence in a weaker sense.

\begin{assumption} (Weaker moment condition on the initial values)\label{manuela}
\begin{equation*}
E[V(X_{0})+V(X_{1})]<\infty.
\end{equation*}
\end{assumption}

\begin{theorem}\label{limit2} Let Assumptions \ref{lyapunov}, \ref{small} and \ref{stary} be in force. 
Assume
\begin{equation}\label{psota2}
r_3(0)+r_4(0)<\infty.
\end{equation} 
Then there is a probability
$\mu_*$ on $\mathcal{X}$ such that $\mu_t\to \mu_*$ in total variation as $t\to\infty$ for every
$X_{0}$ satisfying Assumption \ref{manuela}.
More precisely,
\begin{equation}\label{ob}
\rho_0(\mu_t,\mu_*)\leq C[r_3(t)+r_4(t)],\ t\in\mathbb{N},
\end{equation}
for some constant $C>0$. 
\end{theorem}

Clearly, Assumption \ref{init} implies Assumption \ref{manuela} and \eqref{psota1} implies \eqref{psota2}.
Next, ergodic theorems corresponding to Theorems \ref{limit} and \ref{limit2} are stated.

\begin{theorem}\label{lln} 
Let Assumptions \ref{lyapunov}, \ref{small}, \ref{init} and \ref{stary}  
be in force, but with $R(n):=8K(n)/\lambda(n)$, $n\in\mathbb{N}$ 
in Assumption \ref{small}. Let $Y$ be an ergodic process.
Let $\phi\in \Phi(V^{\delta})$ 
for some $0< \delta\leq 1/2$. Assume 
$$
r_1(0)+r_2(0)<\infty{}
$$
and
\begin{equation}\label{gabba}
\left(\frac{K(g(t))}{\lambda(g(t))}\right)^{2\delta}\frac{\pi(t)}{t}\to 0,\ t\to\infty.
\end{equation}
Then 
\begin{equation}\label{matee1}
\frac{\phi(X_1)+\ldots+\phi(X_t)}{t}\to\int_{\mathcal{X}} \phi(z)\mu_*(dz),\ t\to\infty
\end{equation}
holds in $L^{p}$ for each $p<1/\delta$. (Here $\mu_*$ is the same as in Theorem \ref{limit} above.) 
\end{theorem}

%The rate of convergence in \eqref{matee1} can be estimated, see the
%proof of Theorem \ref{lln} in Section \ref{prof1} below. 

We can weaken our assumptions for bounded $\phi$.

\begin{theorem}\label{llnn} 
Let Assumptions \ref{lyapunov}, \ref{small}, \ref{stary} and \ref{manuela} be in force, but with  
$R(n):=8K(n)/\lambda(n)$, $n\in\mathbb{N}$ in Assumption \ref{small}. Let $Y$ be an ergodic process. Assume 
$$
r_3(0)+r_4(0)<\infty
$$
and
\begin{equation}\label{gadde}
\frac{\pi(t)}{t}\to 0,\ t\to\infty.
\end{equation} 
Then for each $\phi\in\Phi(1)$ the convergence \eqref{matee1} holds in $L^{p}$ for all $p\geq 1$.
\end{theorem}

\begin{remark}\label{nemes}{\rm In Theorems \ref{lln} and \ref{llnn} above, we require a slight strenghtening
of Assumption \ref{small} by imposing \eqref{mey} with $R(n)=8K(n)/\lambda(n)$ instead
of $R(n)=4K(n)/\lambda(n)$. 

Condition \eqref{gadde} is closely related to the condition
$r_3(0)<\infty$ but none of the two implies the other. Indeed, fix $g(t):=t$. 
Choose $\lambda$ constant and $\alpha(t):=\sqrt{\ln(t)}/{t}$, $t\geq 4$. Then $\pi(t)/t\to 0$
but $r_3(0)=\infty$. Conversely, let $\alpha:=1/3$ and $\lambda(t)=\frac{12\ln(t)}{t}$.
Then $r_3(0)<\infty$ but $\pi(t)/t$ tends to a positive constant.}
\end{remark}

\begin{example}\label{mbaker} {\rm Let $Y$ be strongly stationary 
$\mathbb{R}^d$-valued with
$E|Y_0|^p<\infty$, $p\geq 1$. Let Assumptions \ref{lyapunov} and \ref{small}
hold with $K(\cdot)$ having at most polynomial growth
(i.e. $K(n)\leq c n^{b}$ with some $c,b>0$) and $\alpha(\cdot)$, $\lambda(\cdot)$ having 
at most polynomial decay
(i.e. $\alpha(n)\geq c n^{-b}$ with some $c,b>0$, similarly for $\lambda$). 
Let Assumption \ref{init} hold. Then Remark \ref{l+} shows (choosing $\chi$ small and $r$ large)
that Theorems \ref{limit} and \ref{lln} apply.}
\end{example}

\section{Difference equations 
in Gaussian environments}\label{dif}

In this section we present examples of processes $X$ that satisfy a
difference equation, modulated by the process $Y$.  
We do not aim at a high degree of generality but prefer to illustrate the power
of the results in Section \ref{mar} in some easily tractable cases. 
We stress that, as far as we know, none of these results follow from the
existing literature.

We fix $\mathcal{Y}=\mathbb{R}^d$ for some $d$ and
$\mathcal{X}=\mathbb{R}$. We also fix a $\mathcal{Y}$-valued zero-mean Gaussian stationary process $Y_t$, $t\in\mathbb{Z}$.
We set $\Vert y\Vert=\lceil|y|\rceil$, $y\in\mathcal{Y}$ as in Example \ref{ppp} above. We will exclusively use $V(x)=|x|$, 
$x\in\mathbb{R}$ in the examples below.

\begin{remark}\label{gagga}
{\rm Let $\xi_t$, $t\in\mathbb{Z}$ be a zero-mean $\mathbb{R}$-valued
stationary Gaussian process with unit variance.
It is well-known that in this case \begin{equation}\label{port}
E\zeta_t\leq \sqrt{2\ln(t)}\leq \sqrt{2t},\ t\geq 1
\end{equation}
holds for $\zeta_t:=\max_{1\leq i\leq t}\xi_i$. Furthermore, for all $a>0$,
\begin{equation}\label{tv}
P(\zeta_t-E\zeta_t\geq a)\leq e^{-a^2/2},
\end{equation}
see \cite{tsirelson,vitale}. Applying \eqref{tv} with $a=\sqrt{2t}$ and
then proceeding analogously with the process $-\xi$, it follows from \eqref{port} that
\begin{equation*}
P\left(\max_{1\leq i\leq t}|\xi_i|\geq 2\sqrt{2t}\right)\leq 2e^{-t}.
\end{equation*}
Applying these observations to every coordinate of $Y$, it follows that Assumption \ref{stary} holds for the process $Y$ with the choice 
$g(k)=\lceil c_1 \sqrt{k}\rceil$,
$\ell(k)=\exp(-c_2 k)$ for some $c_1,c_2>0$ and thus $r_4(n)$ decreases at a geometric rate as $n\to\infty$.

More generally, choosing $a=t^{b}$ with some $b>0$, Assumption \ref{stary} holds for $Y$  
with the choice $g(k)=\lceil c_1 k^{b} \rceil$,
$\ell(k)=\exp(-c_2 k^{2b})$, by updating \eqref{port} and \eqref{tv}. }
\end{remark}

We assume throughout this section that 
$\varepsilon_t$, $t\in\mathbb{N}$ is an $\mathbb{R}$-valued i.i.d. sequence, independent of $Y_t$, $t\in\mathbb{Z}$;
$E|\varepsilon_0|^2<\infty$ and the law of $\varepsilon_0$
has an everywhere positive density $f$ with respect to the Lebesgue measure, which is even and non-increasing on 
$[0,\infty)$.
All these hypotheses could clearly be weakened/modified, we just try to stay as simple as possible.

\begin{example}\label{hummel}{\rm First we investigate the effect of the ``contraction coefficient'' $\lambda$
in \eqref{lyapi}.
Let $d:=1$. Let $0<\underline{\sigma}\leq \overline{\sigma}$ be constants and 
$\sigma:\mathbb{R}\times\mathbb{R}\to [\underline{\sigma},\overline{\sigma}]$
a measurable function. Let furthermore $\Delta:\mathbb{R}\to (0,1]$ be even and non-increasing on $[0,\infty)$,
for which we will develop conditions on the way.
We stipulate that the tail of $f$ is not too thin: it is at least as thick as that of a Gaussian variable, that is,
\begin{equation}\label{gsu}
f(x)\geq e^{-sx^2},\ x\geq 0,
\end{equation}
for some $s>0$. 

We assume that the dynamics of $X$ is given by
\begin{equation*}
X_0:=0,\ X_{t+1}:=(1-\Delta(Y_t))X_t+\sigma(Y_t,X_t)\varepsilon_{t+1},\ t\in\mathbb{N}.
\end{equation*}
We will find $K(\cdot),\lambda(\cdot),\alpha(\cdot)$ such that Assumptions \ref{lyapunov} and 
\ref{small} hold and give an estimate for the rate $r_3(n)$ appearing in \eqref{ob}. (Note that
we already have estimates for the rate $r_4(n)$ from Remark \ref{gagga}.)

The density of $X_1$ conditional to $X_0=x$, $Y_0=y$ (w.r.t. the Lebesgue measure) is easily seen to be
\[
h_{x,y}(z):=f\left(\frac{z-(1-\Delta(y))x}{\sigma(y,x)}\right)\frac{1}{\sigma(y,x)},\ z\in\mathbb{R}.
\]
Fixing $\eta>0$, we can estimate
\begin{eqnarray*}
\inf_{x,z\in [-\eta,\eta]}h_{x,y}(z)\geq 
f\left(\frac{2\eta}{\underline{\sigma}}\right)\frac{1}{\overline{\sigma}}=:m(\eta),
\end{eqnarray*}
and $m(\cdot)$ does not depend on $y$.
Define the probability measures $$
\tilde{\nu}_{\eta}(A):=\frac{1}{2\eta}\mathrm{Leb}(A\cap [-\eta,\eta]),\ A\in\mathfrak{B}.
$$
It follows that 
\[
Q(y,x,A)\geq 2\eta m(\eta) \tilde{\nu}_{\eta}(A),\ A\in\mathfrak{B},
\]
for all $x\in [-\eta,\eta]$, $y\in\mathbb{R}$.
Notice that
\[
[Q(y)V](x)\leq (1-\Delta(y))V(x) + \overline{\sigma}E|\varepsilon_0|\leq (1-\Delta(y))V(x)+K,
\]
where $K:=\max\{\overline{\sigma}E|\varepsilon_0|,1\}$. 
Then Assumption \ref{lyapunov} holds with $A_n=\{x\in\mathbb{R}:\, |x|\leq n\}$, $\lambda(n)=\Delta(n)$ and $K(n)=K$, $n\geq 1$.
(Here and in the sequel we use the index set $\mathbb{N}\setminus\{0\}$ instead of $\mathbb{N}$ for convenience.)

Let $\eta:=\tilde{R}(y):=4K/\Delta(y)$, $y\in\mathcal{Y}$ and $R(n)=\tilde{R}(n)$,
$n\in\mathbb{N}$. We note that $\tilde{R}(y)$ is 
defined for every $y\in\mathcal{Y}$ while $R(n)$ is defined for every $n\in\mathbb{N}$, this is
why we keep different notations for these two functions here and also in the subsequent examples.
We can conclude, using the tail bound \eqref{gsu} that
\[
Q(y,x,A)\geq \frac{8Km(\tilde{R}(y))}{\Delta(y)}\nu_{\tilde{R}(y)}(A)\geq 
 \frac{e^{-c_3 \tilde{R}^2(y)}}{\Delta(y)}\nu_{\tilde{R}(y)}(A), 
\]
for all $A\in\mathfrak{B}$, with some $c_3>0$ so \eqref{mey} in Assumption \ref{small} holds with
\[
\alpha(n):=e^{-c_3 {R}^2(n)}/\Delta(0),\ n\geq 1,
\]
and $\nu_n=\tilde{\nu}_{R(n)}$.
Now let the function $\Delta$ be such that $\Delta(y):= 1$ for $0\leq y<3$ and
$\Delta(y)\geq 1/(\ln(y))^{\delta}$ with some $\delta>0$, for all $y\geq 3$.
We obtain from the previous estimates and from Remark \ref{gagga} with $g(k)=\lceil c_1 \sqrt{k}\rceil$ that
\[
\lambda(g(k))\alpha(g(k))\geq e^{-c_4\ln^{2\delta}(k)},
\]
with some $c_4>0$.
When $\delta<1/2$, this leads to estimates on the terms of $r_3(n)$ which guarantee $r_3(0)<\infty$. 

If instead of \eqref{gsu} we assume 
\[
f(x)\geq e^{-sx},\ x\geq 0,
\] 
then $r_3(0)<\infty$ follows whenever $\delta<1$. This shows nicely the interplay between
the feasible fatness of the tail of $f$ and the strength of the mean-reversion $\Delta(\cdot)$.} 
\end{example}

\begin{example}
{\rm Again, let $d:=1$, $X_0:=0$ and
\begin{equation*}
X_{t+1}:=(1-\Delta) X_t +\sigma(Y_t,X_t)\varepsilon_{t+1},\ t\in\mathbb{N},
\end{equation*}
where $\sigma:\mathbb{R}\times\mathbb{R}\to (0,\infty)$ is a measurable function and $0\leq \Delta<1$ is a constant.
We furthermore assume that 
\begin{equation*}
c_5 G(y)\leq \sigma(y,x)\leq c_6 G(y),\ x\in\mathbb{R},
\end{equation*}
with some even function $G:\mathbb{R}\to (0,\infty)$ that is nondecreasing on $[0,\infty)$ and with constants $c_5,c_6>0$.
We clearly have \eqref{lyapi} with $\lambda(n)=\Delta$, $n\in\mathbb{N}$
(i.e.\ $\lambda(\cdot)$ is constant)
and $A_n=\{x\in\mathbb{R}:\ |x|\leq n\}$, $K(n):=\tilde{K}(n)$, $n\in\mathbb{N}$ where 
$\tilde{K}(y)= c_6 G(y)E\vert\varepsilon_0\vert$, $y\in\mathbb{R}$. 
Taking $\tilde{R}(y)= 4\tilde{K}(y)/\Delta$, $y\in\mathbb{R}$, estimates as in Example \ref{hummel}
lead to
\[
Q(y,x,A)\geq  2\tilde{R}(y)f\left(\frac{2\tilde{R}(y)}{c_5 G(y)}\right)\frac{1}{c_6 G(y)}\tilde{\nu}_{\tilde{R}(y)}(A)\geq c_7 
\tilde{\nu}_{\tilde{R}(y)}(A), 
\]
for all $A\in\mathfrak{B}$ with some fixed constant $c_7>0$, where $\tilde{\nu}_{\tilde{R}(y)}(\cdot)$ is the normalized
Lebesgue measure restricted to $C(\tilde{R}(y))$, as in Example \ref{hummel} above, so setting $R(n)=\tilde{R}(n)$,
$n\in\mathbb{N}$, we can choose $\nu_n=\tilde{\nu}_{R(n)}$ and $\alpha(\cdot)$ a positive constant. 

%If $G(\cdot)$ is constant then so is $K(\cdot)$ so Theorems \ref{petrie} and \ref{lln2} apply. More generally,
Assume e.g., $G(y)\leq C[1+|y|^q]$, $y\geq 0$ with some $C,q>0$ and choose $g(k)=\lceil c_1\sqrt{k}\rceil$, $\ell(k)=\exp(-c_2 k)$, 
as discussed in Remark \ref{gagga}. Then Theorems \ref{limit} and \ref{lln} apply.}
\end{example}

\begin{example} {\rm We now investigate a discrete-time model for financial time series,
inspired by the ``fractional
stochastic volatility model'' of \cite{cr,gjr}. 

Let $w_t$, $t\in\mathbb{Z}$ and $\varepsilon_t$, $t\in\mathbb{N}$
be two sequences of i.i.d. random variables such that the two sequences
are also independent. Assume that $w_t$ are Gaussian. We define the (causal) infinite moving average process
$$
\xi_t:=\sum_{j=0}^{\infty} a_jw_{t-j},\ t\in\mathbb{Z}.
$$
This series is almost surely convergent whenever $\sum_{j=0}^{\infty} a_j^2<\infty$.
We take $d=2$ here and the random environment will be the $\mathcal{Y}=\mathbb{R}^2$-valued process $Y_t=(w_t,\xi_t)$,
$t\in\mathbb{Z}$. 

We imagine that $\xi_t$ describes the \emph{log-volatility} of an asset in a financial market. 
It is reasonable to assume that $\xi$ is a Gaussian linear process (see \cite{gjr}
where the related continuous-time models are discussed in detail).

Let us now consider the $\mathbb{R}$-valued process $X$ which will describe
the \emph{increment of the log-price} of the given asset. Assume that $X_0:=0$,
\[
X_{t+1}=(1-\Delta) X_t+\rho e^{\xi_t}w_t +\sqrt{1-\rho^2}e^{\xi_t}\varepsilon_{t+1},\ 
t\in\mathbb{N},
\]
with some $-1<\rho<1$, $0<\Delta\leq 1$. The logprice is thus jointly driven by 
the noise sequences $\varepsilon_t$, $w_t$. The parameter
$\Delta$ is responsible for the autocorrelation of $X$ ($\Delta$ is typically close to $1$). The parameter $\rho$
controls the correlation of the price and its volatility. This is found to be non-zero (actually, negative)
in empirical studies, see \cite{cont}, hence it is important to include $w_t$, $t\in\mathbb{Z}$ both in the dynamics of $X$ and
in that of $Y$. We take $A_n=\{y=(w,\xi)\in\mathbb{R}^2:\ |y|\leq n\}$, $n\in\mathbb{N}$.

Notice that
$$
|X_{1}|\leq (1-\Delta) |X_0| + [|w_0|+|\varepsilon_{1}|]e^{\xi_t}
$$
hence
%It follows that if $|x|\geq c_{8} e^{\xi}(1+|w|)$ for some suitably large $c_{8}>0$ then
%$$
%E[V(X_{1})\vert X_0=x,\ Y_0=(w,\xi)] \leq (1-\Delta/2) V(x),
%$$
%which implies 
$$
E[V(X_{1})\vert X_0=x,\ Y_0=(w,\xi)] \leq (1-\Delta) V(x)+c_{8}e^{\xi}(1+|w|)
$$
for {all} $x\in\mathbb{R}$, with some $c_{8}>0$,
i.e.\ Assumption \ref{lyapunov} holds with $\lambda(n)=\lambda:=\Delta$ and 
$K(n)=c_8 e^n(1+n)$.

We now turn our attention to Assumption \ref{small}. Denote the density of the law of
$X_1$ conditional to $X_0=x$, $Y_0=(w,\xi)$ with respect to the Lebesgue measure by 
$h_{x,w,\xi}(z)$, $z\in\mathbb{R}$.
For $x,z\in [-\eta,\eta]$ we clearly have 
\begin{equation}\label{matroz}
h_{x,w,\xi}(z)\geq f\left(\frac{2\eta+e^{\xi}|w|}{e^{\xi}\sqrt{1-\rho^2}}\right)\frac{1}{e^{\xi}\sqrt{1-\rho^2}}.
\end{equation}

We assume from now on that $f$, the density of $\varepsilon_0$ satisfies 
\begin{equation*}
f(x)\geq s/(1+x)^{\chi},\ x\geq 0
\end{equation*} 
with some $s>0$, $\chi>3$, this is reasonable as $X_t$ has fat tails according to empirical studies, 
see \cite{cont}. At the same time, Assumption \ref{init} can also be satisfied for such a choice
of $f$. 

Define $\tilde{K}(y):=e^{\xi}(1+|w|)$ and $\tilde{R}(y):=4\tilde{K}(y)/\lambda$, for $y=(w,\xi)\in\mathbb{R}^2$. 
Use \eqref{matroz} to obtain, as in Example \ref{hummel} above,
$$
Q(y,x,A)\geq \frac{c_{9}}{(1+|w|)^{\chi}}\frac{1}{e^{\xi}}2\tilde{R}(y)\tilde{\nu}_{\tilde{R}(y)}(A)\geq \frac{c_{10}}{(1+|w|)^{\chi-1}}
\tilde{\nu}_{\tilde{R}(y)}(A), 
$$
with fixed constants $c_{9},c_{10}>0$,
where $\tilde{\nu}_\eta$ is the normalized Lebesgue measure restricted to $[-\eta,\eta]$.
Set $R(n)=\tilde{R}((n,n))$, $n\geq 1$. Then Assumption \ref{small} holds with
$$
\alpha(n)=\frac{c_{10}}{(1+n)^{\chi-1}},\ n\geq 1
$$
and $\nu_{n}=\tilde{\nu}_{R(n)}$.
Recalling the end of Remark \ref{gagga}, and choosing $b>0$ small enough we can conclude
that Theorems \ref{limit} and \ref{lln} apply to this stochastic volatility model.}
\end{example}

We stress that only a small fraction of relevant examples has been presented above, favouring
simplicity. The results of Section \ref{mar} clearly apply in much greater generality.

\section{Proofs of stochastic stability}\label{prof}

Denote by $\zeta$ the law of $\mathbf{Y}:=(Y_{t})_{t\in\mathbb{Z}}\in \mathfrak{Y}:=\mathcal{Y}^{\mathbb{Z}}$ on the Borel sets 
$\mathfrak{T}$ of
$\mathfrak{Y}$. By the measure decomposition theorem there is a probabilistic kernel $\tilde{\mu}_{0}:\mathfrak{Y}\times\mathfrak{B}\to{}
[0,1]$ such that 
\begin{equation}\label{senegal}
\mu_{0}(A)=\int_{\mathfrak{Y}}\tilde{\mu}_{0}(\mathbf{y},A)\zeta(d\mathbf{y}),\ A\in\mathfrak{B}.
\end{equation}
For each $\mathbf{y}\in\mathfrak{Y}$, we will denote by $\tilde{\mu}_{0}(\mathbf{y})$ the 
probability $A\to \tilde{\mu}_{0}(\mathbf{y},A)$, $A\in\mathfrak{B}$
in the sequel.

Clearly, Assumption \ref{init} is equivalent to 
\begin{equation}\label{johnson}
E\left[\int_{\mathcal{X}} V^{2}(z)[\tilde{\mu}_0(\mathbf{Y})+[Q(Y_0)\tilde{\mu}_0(\mathbf{Y})]](dz)\right]<\infty,
\end{equation}
and Assumption \ref{manuela} is equivalent to
\begin{equation}\label{johnson1}
E\left[\int_{\mathcal{X}} V(z)[\tilde{\mu}_0(\mathbf{Y})+[Q(Y_0)\tilde{\mu}_0(\mathbf{Y})]](dz)\right]<\infty.
\end{equation}

We first present a result of \cite{hm} 
(see also the related ideas in \cite{hms}) which will be used below.

\begin{lemma}\label{alap}
Let $L:\mathcal{X}\times\mathfrak{B}\to [0,1]$ be a probabilistic kernel such that
\[
LV(x)\leq \gamma V(x)+K,\ x\in\mathcal{X},
\]
for some $0\leq \gamma<1$, $K>0$. Let
$C:=\{x\in\mathcal{X}:\, V(x)\leq R \}$ for some $R>2K/(1-\gamma)$. 
Let us assume that there is a probability $\nu$ on $\mathfrak{B}$
such that
\[
\inf_{x\in C} L(x,A)\geq \alpha \nu(A),\ A\in\mathfrak{B},
\]
for some $\alpha>0$. Then for each $\alpha_0\in (0,\alpha)$ and for 
$\gamma_0:=\gamma + 2K/R$, 
\[
\rho_{\beta}(L\mu_1,L\mu_2)\leq \max\left\{1-(\alpha-\alpha_0),\frac{2+R\beta\gamma_0}{2+R\beta}\right\}
\rho_{\beta}(\mu_1,\mu_2),\ \mu_1,\mu_2\in\mathcal{P}_V,
\]
holds for $\beta=\alpha_0/K$.\hfill $\Box$
\end{lemma}

For the proof, see Theorem 3.1 in \cite{hm}.
Next comes an easy corollary.

\begin{lemma}\label{lett} 
Let $L:\mathcal{X}\times\mathfrak{B}\to [0,1]$ be a probabilistic kernel such that
\begin{equation}\label{est1}
LV(x)\leq (1-\lambda) V(x)+K,\ x\in\mathcal{X},
\end{equation}
for some $0<\lambda\leq 1/3$, $K>0$. Let
$C:=\{x\in\mathcal{X}:\, V(x)\leq R \}$ with $R:=4K/\lambda$. 
Assume that there is a probability $\nu$ on $\mathfrak{B}$ such that
\begin{equation}\label{est2}
\inf_{x\in C} L(x,A)\geq \alpha \nu(A),\ A\in\mathfrak{B},
\end{equation}
for some $0<\alpha\leq 1/3$.
Then
\[
\rho_{\beta}(L\mu_1,L\mu_2)\leq \left(1-\frac{\alpha\lambda}{2}\right)\rho_{\beta}(\mu_1,\mu_2),\ \mu_1,\mu_2\in\mathcal{P}_V,
\]
holds for $\beta=\alpha/2K$. 
\end{lemma}
\begin{proof} 
Choose $\gamma:=1-\lambda$, and let $\alpha_0:=\alpha/2$.
Note that $1-(\alpha-\alpha_0)= 1-\alpha/2$ and $R\beta=4\alpha_0/(1-\gamma)$
holds for $\beta=\alpha_0/K$.
Applying Lemma \ref{alap}, we estimate
\begin{eqnarray*}
\rho_{\beta}(L\mu_1,L\mu_2) &\leq&\\
\max\left\{
1-(\alpha-\alpha_0),\frac{2+R\beta\gamma_0}{2+R\beta}\right\}\rho_{\beta}(\mu_1,\mu_2)
&=&\\
\max\left\{
1-\alpha/2,1-\frac{4\alpha_0(1-\gamma_0)/(1-\gamma)}{2+4\alpha_0/(1-\gamma)}\right\}\rho_{\beta}(\mu_1,\mu_2). & &
\end{eqnarray*}
Here 
\[
\frac{4\alpha_0(1-\gamma_0)/(1-\gamma)}{2+4\alpha_0/(1-\gamma)}=
\frac{\alpha_0\lambda}{\lambda+2\alpha_0}\geq \alpha_0\lambda
\]
and we get the statement since $\alpha/2\geq \alpha_0\lambda$.
\end{proof}

We introduce some important notation now. Let us consider $\mathfrak{Y}$ equipped by its Borel sigma-algebra $\mathfrak{T}$. 
If $(\mathbf{y},A)\to L(\mathbf{y},A)$, $\mathbf{y}\in\mathfrak{Y}$, $A\in\mathfrak{B}$
is a (not necessarily probabilistic) kernel and $Z$ is a $\mathfrak{Y}$-valued random variable then we define
a measure $\mathcal{E}[L(Z)](\cdot)$ on $\mathfrak{B}$ via
\begin{equation}\label{proba}
\mathcal{E}[L(Z)](A):=E[L(Z,A)],\ A\in\mathfrak{B}.
\end{equation}

We will use the following trivial inequalities in the sequel:
\begin{equation}\label{tvo}
\rho_0(\cdot)\leq 2,\quad 
\rho_0(\cdot)\leq \rho_{\beta}(\cdot)\leq \rho_1(\cdot)\leq \left(1+\frac{1}{\beta}\right)\rho_{\beta}(\cdot),\
0<\beta\leq 1.
\end{equation}

\begin{proof}[Proof of Theorem \ref{limit}.] For later use, we define the $\mathfrak{Y}$-valued random variables 
$\hat{\mathbf{Y}}_{n}:=(Y_{n+j})_{j\in\mathbb{Z}}$, for each $n\in\mathbb{Z}$.
Note that $\mathbf{Y}=\hat{\mathbf{Y}}_{0}$.
Fix $\mathbf{y}:=(y_j)_{j\in\mathbb{Z}}\in \mathfrak{Y}$ for the moment.
Set $\hat{\mathbf{y}}_{n}:=(y_{n+j})_{j\in\mathbb{Z}}$, for each $n\in\mathbb{Z}$. Again, $\mathbf{y}=\hat{\mathbf{y}}_{0}$.
%Furthermore, let $\mathbf{y}_n:=(y_0,y_{-1},\ldots,y_{-n+1})$, $n\geq 1$ and 
Define 
\begin{equation}\label{dafie}
\mu_0(\mathbf{y}):=\tilde{\mu}_{0}(\mathbf{y}),\ \mu_n(\mathbf{y}):=Q(y_0)Q(y_{-1})\ldots Q(y_{-n+1})
\tilde{\mu}_{0}(\hat{\mathbf{y}}_{-n+1}),\ n\geq 1.
\end{equation}
Here $Q(y)$ is the operator acting on probabilities which is described in \eqref{taho} above but, 
instead of $L(x,A)$, with the kernel $Q(y,x,A)$.
Fix $n\geq 1$ and denote $\bar{y}_n:=\max_{-n+1\leq j\leq 0}\Vert y_j\Vert$. Since
\[
\alpha(\Vert y_j\Vert)\geq \alpha(\bar{y}_n),\ \lambda(\Vert y_j\Vert)\geq 
\lambda(\bar{y}_n),\ 
K(\Vert y_j\Vert)\leq 
K(\bar{y}_n),
\]
for each $-n+1\leq j\leq 0$, \eqref{est1} and \eqref{est2} hold for $L=Q(y_j)$,
$j=-n+1,\ldots,0$ with $K=K(\bar{y}_n)$, $\lambda=\lambda(\bar{y}_n)$
and $\alpha=\alpha(\bar{y}_n)$. An $n$-fold application of 
Lemma \ref{lett} implies that, for $\beta=\alpha(\bar{y}_n)/2K(\bar{y}_n)$,
\[
\rho_{\beta}(\mu_n(\mathbf{y}),\mu_{n+1}(\mathbf{y}))\leq (1-\alpha(\bar{y}_n)\lambda(\bar{y}_n)/2)^{n}
\rho_{\beta}(\tilde{\mu}_0(\hat{\mathbf{y}}_{-n+1}),Q(y_{-n})\tilde{\mu}_{0}
(\hat{\mathbf{y}}_{-n})).
\]
By \eqref{tvo} and by $K(\cdot)/\alpha(\cdot)\geq 1$, 
\begin{eqnarray}\nonumber
\rho_{1}(\mu_n(\mathbf{y}),\mu_{n+1}(\mathbf{y})) &\leq&\\
\nonumber \left(1+\frac{2K(\bar{y}_n)}{\alpha(\bar{y}_n)}\right)
(1-\alpha(\bar{y}_n)\lambda(\bar{y}_n)/2)^n \rho_{\beta}(\tilde{\mu}_0(\hat{\mathbf{y}}_{-n+1}),Q(y_{-n})\tilde{\mu}_{0}
(\hat{\mathbf{y}}_{-n}))
&\leq& \\ 
\frac{3K(\bar{y}_n)}{\alpha(\bar{y}_n)}
(1-\alpha(\bar{y}_n)\lambda(\bar{y}_n)/2)^n \rho_1(\tilde{\mu}_0(\hat{\mathbf{y}}_{-n+1}),Q(y_{-n})\tilde{\mu}_{0}
(\hat{\mathbf{y}}_{-n})). & &\label{palfus}
\end{eqnarray}

%Now let $\mathbf{Y}_n:=(Y_0,Y_{-1},\ldots,Y_{-n+1})$. 

We thus arrive at
\begin{eqnarray*}
& & E[\rho_1(\mu_n(\mathbf{Y}),\mu_{n+1}(\mathbf{Y}))]\nonumber\\
&\leq&
3E\left[\frac{K(M_n)}{\alpha(M_n)}
(1-\alpha(M_n)\lambda(M_n)/2)^{n}\rho_1(\tilde{\mu}_0(\hat{\mathbf{Y}}_{-n+1}),
Q(Y_{-n})\tilde{\mu}_0(\hat{\mathbf{Y}}_{-n}))\right],\label{rr}
\end{eqnarray*}
using the notation $M_n:=\max_{-n+1\leq i\leq 0}\Vert Y_i\Vert$.
We now estimate the expectation on the right-hand side of \eqref{palfus} separately on the 
events $\{M_n\geq g(n)\}$ and $\{M_n< g(n)\}$. 
Note that
\begin{eqnarray*}
& & E\left[\frac{K(M_n)}{\alpha(M_n)}\left(1-\frac{\alpha(M_n)\lambda(M_n)}{2}\right)^n
\rho_1(\tilde{\mu}_0(\hat{\mathbf{Y}}_{-n+1}),
Q(Y_{-n})\tilde{\mu}_0(\hat{\mathbf{Y}}_{-n}))1_{\{|M_n|\geq g(n)\}}\right]\\
&\leq& \sum_{k=n}^{\infty} \frac{K(g(k+1))}{\alpha(g(k+1))}\left(1-\frac{\alpha(g(k+1))\lambda(g(k+1))}{2}\right)^n
E\left[\rho_1(\tilde{\mu}_0(\hat{\mathbf{Y}}_{-n+1}),
Q(Y_{-n})\tilde{\mu}_0(\hat{\mathbf{Y}}_{-n})) 1_{\{g(k+1)>|M_n|\geq g(k)\}}\right]\\
&\leq& \sum_{k=n}^{\infty} \frac{K(g(k+1))}{\alpha(g(k+1))}\left(1-\frac{\alpha(g(k+1))\lambda(g(k+1))}{2}\right)^n
E\left[\rho_1(\tilde{\mu}_0(\hat{\mathbf{Y}}_{-n+1}),
Q(Y_{-n})\tilde{\mu}_0(\hat{\mathbf{Y}}_{-n}))1_{\{|M_n|\geq g(k)\}}\right].
\end{eqnarray*}
Hence
\begin{eqnarray*}
& & \sum_{m=n}^{\infty} E[\rho_1(\mu_m(\mathbf{Y}),\mu_{m+1}(\mathbf{Y}))]\\
&\leq& 3\sum_{m=n}^{\infty} 
\frac{K(g(m))}{\alpha(g(m))}e^{-\frac{m}{2}\alpha(g(m))\lambda(g(m))}E\left[\rho_1(\tilde{\mu}_0(\hat{\mathbf{Y}}_{-m+1}),
Q(Y_{-m})\tilde{\mu}_0(\hat{\mathbf{Y}}_{-m})) 1_{\{|M_m|< g(m)\}}\right]\\
&+& 3\sum_{m=n}^{\infty}\sum_{k=m}^{\infty} \frac{K(g(k+1))}{\alpha(g(k+1))}\left(1-\frac{\alpha(g(k+1))\lambda(g(k+1))}{2}\right)^m 
E\left[\rho_1(\tilde{\mu}_0(\hat{\mathbf{Y}}_{-m+1}),
Q(Y_{-m})\tilde{\mu}_0(\hat{\mathbf{Y}}_{-m}))1_{\{|M_m|\geq g(k)\}}\right]\\
&\leq& 3\sum_{m=n}^{\infty} 
\frac{K(g(m))}{\alpha(g(m))}e^{-\frac{m}{2}\alpha(g(m))\lambda(g(m))}E\left[\rho_1(\tilde{\mu}_0(\hat{\mathbf{Y}}_{-m+1}),
Q(Y_{-m})\tilde{\mu}_0(\hat{\mathbf{Y}}_{-m}))\right]\\
&+& 3\sum_{k=n}^{\infty}\sum_{m=n}^k  
\frac{K(g(k+1))}{\alpha(g(k+1))}\left(1-\frac{\alpha(g(k+1))\lambda(g(k+1))}{2}\right)^m
E\left[\rho_1(\tilde{\mu}_0(\hat{\mathbf{Y}}_{-m+1}),
Q(Y_{-m})\tilde{\mu}_0(\hat{\mathbf{Y}}_{-m}))1_{\{|M_k|\geq g(k)\}}\right]
\\
&\leq& 3\sum_{m=n}^{\infty} 
\frac{K(g(m))}{\alpha(g(m))}e^{-\frac{m}{2}\alpha(g(m))\lambda(g(m))}E\left[\rho_1(\tilde{\mu}_0(\hat{\mathbf{Y}}_{-m+1}),
Q(Y_{-m})\tilde{\mu}_0(\hat{\mathbf{Y}}_{-m}))\right]\\
&+& 6\sum_{k=n}^{\infty}  
\frac{K(g(k+1))}{\alpha^2(g(k+1))\lambda(g(k+1))}E^{1/2}\left[\rho_1^{2}(\tilde{\mu}_0(\hat{\mathbf{Y}}_{1}),
Q(Y_{0})\tilde{\mu}_0(\mathbf{Y}))\right]
P^{1/2}(|M_k|\geq g(k))\\
&\leq& 3E\left[\rho_1(\tilde{\mu}_0(\hat{\mathbf{Y}}_{1}),
Q(Y_{0})\tilde{\mu}_0(\mathbf{Y}))\right] \sum_{m=n}^{\infty} 
\frac{K(g(m))}{\alpha(g(m))}e^{-\frac{m}{2}\alpha(g(m))\lambda(g(m))}\\
&+& 6E^{1/2}\left[\rho_1^{2}(\tilde{\mu}_0(\hat{\mathbf{Y}}_{1}),
Q(Y_{0})\tilde{\mu}_0(\mathbf{Y}))\right]\sum_{k=n}^{\infty} 
\frac{K(g(k+1))}{\alpha^2(g(k+1))\lambda(g(k+1))}\sqrt{\ell(k)},
\end{eqnarray*}
where we have used $M_{k}\geq M_{m}$ in the second inequality;
the closed form expression for the sum of geometric series and Cauchy-Schwarz in the third inequality;
Assumption \ref{stary} and the fact that the law of 
$\rho_1(\tilde{\mu}_0(\hat{\mathbf{Y}}_{1}),
Q(Y_{0})\tilde{\mu}_0(\mathbf{Y}))$ equals that of $$
\rho_1(\tilde{\mu}_0(\hat{\mathbf{Y}}_{-m+1}),
Q(Y_{-m})\tilde{\mu}_0(\hat{\mathbf{Y}}_{-m})),
$$ 
for each $m$, in the fourth inequality.
 
Recall that
\begin{eqnarray*}
& & E[\rho_1^{2}(\tilde{\mu}_0(\hat{\mathbf{Y}}_{1}),
Q(Y_{0})\tilde{\mu}_0(\mathbf{Y}))]\\
&\leq& E\left[\int_{\mathcal{X}} 2(1+V(z))^{2} [\tilde{\mu}_0(\hat{\mathbf{Y}}_{1})+
Q(Y_{0})\tilde{\mu}_0(\mathbf{Y})](dz) \right]\\
&=& 
E\left[\int_{\mathcal{X}} 2(1+V(z))^{2} [\tilde{\mu}_0(\mathbf{Y})+
Q(Y_{0})\tilde{\mu}_0(\mathbf{Y})](dz) \right]<\infty
\end{eqnarray*}
by \eqref{johnson}. A fortiori, $E[\rho_1(\tilde{\mu}_0(\hat{\mathbf{Y}}_{1}),
Q(Y_{0})\tilde{\mu}_0(\mathbf{Y}))]<\infty$, too.

%Note also that 
%\begin{eqnarray*}
%& & E\rho_1(\mu_0(\mathbf{Y}),\mu_1(\mathbf{Y}))\\
%&=& \int_{\mathcal{X}}(1+V(z))
%\mathcal{E}\left[\left|\mu_1(\mathbf{Y})-\mu_{0}(\mathbf{Y})\right|\right](dz)\\
%&=&  E\left[\int_{\mathcal{X}}(1+V(z))
%\left|\mu_1(\mathbf{Y})-\mu_{0}(\mathbf{Y})\right|(dz)\right]\\
%&\leq& E\left[\int_{\mathcal{X}}(1+V(z))
%(\mu_1(\mathbf{Y})+\mu_{0}(\mathbf{Y}))(dz)\right]\\
%&=& E\left[\int_{\mathcal{X}}(1+V(z))
%(Q(Y_{1})\tilde{\mu}_0(\mathbf{Y})+\tilde{\mu}_{0}(\mathbf{Y}))(dz)\right]<\infty,{}
%\end{eqnarray*}  
%by Assumption \ref{init}.

Now it follows from $r_1(0)+r_2(0)<\infty$ that
\begin{equation}\label{morkonn}
\sum_{n=1}^{\infty}E[\rho_1(\mu_n(\mathbf{Y}),\mu_{n+1}(\mathbf{Y}))]<\infty.
\end{equation}
Consequently, for a.e.\ $\omega$, the sequence $\mu_n(\mathbf{Y}(\omega))$, $n\in\mathbb{N}$
is Cauchy and hence convergent for the metric $\rho_{1}$. Its limit is denoted
by $\mu_{\sharp}(\omega)$. 

For later use, we remark that $\omega\to \int_{\mathcal{X}}\phi(z)\mu_{\sharp}(\omega)(dz)$ is
$\sigma(\mathbf{Y})$-measurable for every $\phi\in \Phi(V)$.  
Hence there is a measurable $\Psi_{\phi}:\mathfrak{Y}\to\mathbb{R}$ such that 
\begin{equation}\label{defie}
\Psi_{\phi}(\mathbf{Y})=\int_{\mathcal{X}}\phi(z)\mu_{\sharp}(dz)\mbox{ a.s.}	
\end{equation}
    
In the sequel we will need 
the definition \eqref{proba} for the kernel 
$(\mathbf{y},A)\to\mu_n(\mathbf{y})(A)$, $\mathbf{y}\in\mathfrak{Y}$,
$A\in\mathfrak{B}$ and for similar kernels.
Notice that, for any measurable function $w:\mathcal{X}\to\mathbb{R}_+$,
\begin{equation}\label{jobim}
\int_{\mathcal{X}} w(z)\,
\left|\mathcal{E}[\mu_n(\mathbf{Y})]-\mathcal{E}[\mu_{n+1}(\mathbf{Y})]\right|(dz)
\leq \int_{\mathcal{X}} w(z)\,
\mathcal{E}\left[\left|\mu_n(\mathbf{Y})-\mu_{n+1}(\mathbf{Y})\right|\right](dz). 
\end{equation}
This is trivial for indicators and then follows for all measurable $w$ in a standard way.
By similar arguments, we also have
$$
\int_{\mathcal{X}}w(z)
\mathcal{E}\left[\left|\mu_n(\mathbf{Y})-\mu_{n+1}(\mathbf{Y})\right|\right](dz)=
E\left[\int_{\mathcal{X}}w(z)
\left|\mu_n(\mathbf{Y})-\mu_{n+1}(\mathbf{Y})\right|(dz)\right].
$$

As easily seen, $\mu_{n}=\mathcal{E}[\mu_n(\mathbf{Y})]$ so we infer that
\begin{eqnarray*}
\rho_1(\mu_n,\mu_{n+1})=\int_{\mathcal{X}}(1+V(z))
\left|\mathcal{E}[\mu_n(\mathbf{Y})]-\mathcal{E}[\mu_{n+1}(\mathbf{Y})]\right|(dz) &\leq &\\
\int_{\mathcal{X}}(1+V(z))
\mathcal{E}\left[\left|\mu_n(\mathbf{Y})-\mu_{n+1}(\mathbf{Y})\right|\right](dz) &=&\\
E\left[\int_{\mathcal{X}}(1+V(z))
\left|\mu_n(\mathbf{Y})-\mu_{n+1}(\mathbf{Y})\right|(dz)\right] &=&\\
E[\rho_1(\mu_n(\mathbf{Y}),\mu_{n+1}(\mathbf{Y}))]. & &
\end{eqnarray*}

Then it follows from \eqref{morkonn} that
\[
\sum_{n=1}^{\infty}\rho_1(\mu_n,\mu_{n+1})<\infty,
\]
so $\mu_n$, $n\geq 0$ is a Cauchy sequence for the complete metric $\rho_1$.
Hence it converges to some probability $\mu_*$ as $n\to\infty$. The claimed convergence rate also follows by the above estimates.

To show uniqueness, let $X_{0}'$ be another initial condition satisfying Assumption \ref{init}, with the corresponding 
$\tilde{\mu}_{0}'(\mathbf{y})$, see \eqref{senegal}.
Defining, just like in \eqref{dafie} above,
$$
\mu_0'(\mathbf{y}):=\tilde{\mu}_{0}'(\mathbf{y}),\ \mu_n'(\mathbf{y}):=
Q(y_0)Q(y_{-1})\ldots Q(y_{-n+1})\tilde{\mu}_{0}'(\hat{\mathbf{y}}_{-n+1}),\ n\geq 1,
$$
the above estimates show that
\begin{eqnarray*}
\rho_{1}(\mathcal{E}[\mu_{n}'(\mathbf{Y})],\mathcal{E}[\mu_{n}(\mathbf{Y})]) &\leq&{}
E[\rho_{1}(\mu_{n}'(\mathbf{Y}),\mu_{n}(\mathbf{Y}))]\\
&\leq& 3E\left[\rho_1(\tilde{\mu}_0(\mathbf{Y}),
\tilde{\mu}_0'(\mathbf{Y}))\right] \sum_{m=n}^{\infty} 
\frac{K(g(m))}{\alpha(g(m))}e^{-\frac{m}{2}\alpha(g(m))\lambda(g(m))}\\
&+& 6E^{1/2}\left[\rho_1^{2}(\tilde{\mu}_0(\mathbf{Y}),
\tilde{\mu}_0'(\mathbf{Y}))\right]\sum_{k=n}^{\infty} 
\frac{K(g(k+1))}{\alpha^2(g(k+1))\lambda(g(k+1))}\sqrt{\ell(k)},
\end{eqnarray*}
which tends to $0$ when $n\to\infty$ since, as before,
$$
E\left[\rho_1^{2}(\tilde{\mu}_0(\mathbf{Y}),
\tilde{\mu}_0'(\mathbf{Y}))\right]<\infty
$$
by Assumption \ref{init}. 
\end{proof}

\begin{remark}\label{wekings}
{\rm Define the probability $\bar{\mu}(A):=E[\mu_{\sharp}(A)]$, $A\in\mathfrak{B}$.
It is clear that, for every $\phi\in\Phi(1)$,
\begin{eqnarray*}
& & \int_{\mathcal{X}}\phi(z)\mu_{*}(dz)\\
&=& \lim_{n\to\infty}\int_{\mathcal{X}}\phi(z)\mu_{n}(dz)\\
&=& \lim_{n\to\infty}\int_{\mathcal{X}}\phi(z)\mathcal{E}[\mu_{n}(\mathbf{Y})](dz)\\
&=& \lim_{n\to\infty}E\left[\int_{\mathcal{X}} \phi(z) \mu_n(\mathbf{Y})(dz)\right]\\
&=& E\left[\int_{\mathcal{X}}\phi(z)\mu_{\sharp}(dz)\right]\\
&=& \int_{\mathcal{X}}\phi(z)\bar{\mu}(dz),	
\end{eqnarray*} 
hence $\bar{\mu}=\mu_{*}$.}
\end{remark}

\begin{remark}{\rm The proof of Theorem \ref{limit} also implies convergence for the ``quenched'' process:
there is a set $\mathfrak{Y}'\subset\mathfrak{Y}$ with $\zeta(\mathfrak{Y}')=1$ (recall that $\zeta$ is
the law of $\mathbf{Y}$) such that, for all $\mathbf{y}\in\mathfrak{Y}'$,
the sequence $\mu_{n}(\mathbf{y})$ converges in $\rho_{1}$ to a limiting probability as $n\to\infty$.}
\end{remark}

\begin{proof}[Proof of Theorem \ref{limit2}.]
Estimates of Theorem \ref{limit} and \eqref{tvo} imply
\[
\rho_{0}(\mu_n(\mathbf{y}_n),\mu_{n+1}(\mathbf{y}_{n+1}))\leq (1-\alpha(\bar{y}_n)\lambda(\bar{y}_n)/2)^{n}
\rho_{1}(\tilde{\mu}_0(\hat{\mathbf{y}}_{-n+1}),Q(y_{-n})\tilde{\mu}_0(\hat{\mathbf{y}}_{-n})).
\]
This leads to
\begin{eqnarray*}
\rho_0(\mu_n,\mu_{n+1})\leq E[\rho_0(\mu_n(\mathbf{Y}),\mu_{n+1}(\mathbf{Y}))] &\leq&\\
(1-\alpha(g(n))\lambda(g(n))/2)^n 
E[\rho_1(\tilde{\mu}_0(\hat{\mathbf{Y}}_{-n+1}),
Q(Y_{-n})\tilde{\mu}_0(\hat{\mathbf{Y}}_{-n}))1_{\{M_n<g(n)\}}]+2P\left(M_n\geq g(n)\right) &\leq&\\
(1-\alpha(g(n))\lambda(g(n))/2)^n 
E[\rho_1(\tilde{\mu}_0(\hat{\mathbf{Y}}_{1}),
Q(Y_{0})\tilde{\mu}_0(\mathbf{Y}))]+2P\left(M_n\geq g(n)\right) 
&\leq &\\
C[e^{-n\alpha(g(n))\lambda(g(n))/2}+{\ell(n)}], & &
\end{eqnarray*}
for some $C>0$, using \eqref{tvo}, Assumption \ref{stary} and \eqref{johnson1}.
The result now follows as in the proof of Theorem \ref{limit} above.
\end{proof}

\begin{remark}
{\rm The convergence rates obtained by our method heavily depend on the choice of the functions $g$ and $\ell$ for which there are
multiple options. Hence no optimality can be claimed. The approach, however, works in many cases where available methods do not.}
\end{remark}
 
\section{$L$-mixing processes}\label{lm}

Let $\mathcal{G}_t$, $t\in\mathbb{N}$ be an increasing sequence of sigma-algebras (i.e. a discrete-time filtration) and let $\mathcal{G}^+_t$, $t\in\mathbb{N}$
be a \emph{decreasing} sequence of sigma-algebras such that, for each $t\in\mathbb{N}$, $\mathcal{G}_t$ is independent of $\mathcal{G}^+_t$.

Let  $W_t$, $t\in\mathbb{N}$ be a real-valued stochastic process. For each $r\geq 1$, introduce
$$
M_r(W):=\sup_{t\in\mathbb{N}} E^{1/r}[|W_t|^r].
$$
For each process $W$ such that $M_1(W)<\infty$ we also define, for each $r\geq 1$, the quantities
$$
\gamma_r(W,\tau):=\sup_{t\geq\tau}E^{1/r}[|W_t-E[W_t|\mathcal{G}_{t-\tau}^+]|^r],\ \tau\in\mathbb{N},\ \Gamma_r(W):=\sum_{\tau=0}^{\infty} \gamma_r(W,\tau).
$$

For some $r\geq 1$, the process $W$ is called 
\emph{$L$-mixing of order $r$} with respect to $(\mathcal{G}_t,\mathcal{G}^+_t)$, $t\in\mathbb{N}$ if
it is adapted to $(\mathcal{G}_t)_{t\in\mathbb{N}}$ and $M_r(W)<\infty$, $\Gamma_r(W)<\infty$. We say that $W$
is \emph{$L$-mixing} if it is $L$-mixing of order $r$ for all $r\geq 1$. This notion of mixing was introduced in \cite{laci1}.

\begin{remark}\label{utu} {\rm It is easy to check that if $W_t$, 
$t\in\mathbb{N}$ is $L$-mixing of order $r$ then
also the process $\tilde{W}_t:=W_t-EW_t$, $t\in\mathbb{N}$ is $L$-mixing of order $r$, moreover,
$\Gamma_r(\tilde{W})=\Gamma_r(W)$ and $M_r(\tilde{W})\leq 2M_r(W)$.}
\end{remark}

The next lemma (Lemma 2.1 of \cite{laci1}) is useful when checking the $L$-mixing property
for a given process.

\begin{lemma}\label{fyffes} Let 
$\mathcal{G}\subset\mathcal{F}$ be a sigma-algebra,
$X$, $Y$ random variables with $E^{1/r}[|X|^r]<\infty$, $E^{1/r}[|Y|^r]<\infty$ with some $r\geq 1$.
If ${Y}$ is $\mathcal{G}$-measurable then
$$
E^{1/r}[|X-E[X\vert\mathcal{G}]|^r]\leq 2E^{1/r}[|X-Y|^r]
$$
holds.\hfill $\Box$ 
\end{lemma}

$L$-mixing is, in many cases, easier to show than other, better-known mixing concepts and
it leads to useful inequalities like Lemma \ref{inek} below. For further
related results, see \cite{laci1}.

\begin{lemma}\label{inek}
For an $L$-mixing process $W$ of order $r\geq 2$ satisfying $E[W_t]=0$, $t\in\mathbb{N}$,
$$
E^{1/r}\left[\left|\sum_{i=1}^N W_i\right|^r\right]\leq C_r N^{1/2} M_r^{1/2}(W)\Gamma_r^{1/2}(W),
$$
holds for each $N\geq 1$ with a constant $C_r$ that does not depend either on $N$ or on $W$.
\end{lemma}
\begin{proof}
This follows from Theorem 1.1 of \cite{laci1}.
\end{proof}

\section{Proofs of ergodicity I}\label{prof1}

Throughout this section let the assumptions of Theorem \ref{lln} be valid:
$Y$ is an ergodic process; let Assumptions \ref{lyapunov} and \ref{init} be in force;
let Assumption \ref{small} hold with $R(n):=8K(n)/\lambda(n)$, $n\in\mathbb{N}$;
assume $r_1(0)+r_2(0)<\infty$ and
\begin{equation*}
\left(\frac{K(g(N))}{\lambda(g(N))}\right)^{2\delta}
\frac{\pi(N)}{N}\to 0,\ N\to\infty.
\end{equation*}

We now present a construction that is crucial for proving Theorem \ref{lln}.
The random mappings $T_t$ in the lemma below serve to provide the coupling
effects that are needed for establishing the $L$-mixing property (see Section \ref{lm} above) for
an auxiliary process ($Z$ below) which will, in turn, lead to Theorem \ref{lln}.
Such a representation with random mappings was used in \cite{bm,bw,gmmtv,msg}. In our 
setting, however, there is also dependence on $y\in\mathcal{Y}$.

For $R\geq 0$, denote by $\mathfrak{C}(R)$ the set of $\mathcal{X}\to\mathcal{X}$ mappings that are constant on 
$C(R)=\{x\in\mathcal{X}:\, V(x)\leq R\}$.

\begin{lemma}\label{t}
There exists a sequence of measurable functions 
$T_t:\mathcal{Y}\times\mathcal{X}\times{\Omega}
\to \mathcal{X}$, $t\geq 1$ such that 
\begin{equation}\label{madrid}
P(T_t(y,x,\omega)\in A)=Q(y,x,A),
\end{equation}
for all $t\geq 1$, $y\in\mathcal{Y}$,
$x\in\mathcal{X}$, $A\in\mathfrak{B}$.
For each $t\geq 1$, let $\mathcal{L}_t$ denote 
the sigma-algebra generated by the random variables 
$T_t(y,x,\cdot),\, x\in\mathcal{X},\, y\in\mathcal{Y}$. These sigma-algebras are independent. 
There are events $J_t(y)\in\mathcal{L}_{t}$, for all $t\geq 1$, $y\in\mathcal{Y}$
such that
\begin{equation}\label{patty}
J_t(y)\subset \{\omega:\, T_t(y,\cdot,\omega)\in\mathfrak{C}(R(\Vert y\Vert))\}\mbox{ and }P(J_t(y))\geq
\alpha(\Vert y\Vert).
\end{equation}
\end{lemma}
\begin{proof}
Let $U_n$, $n\in\mathbb{N}$ be an independent sequence of uniform random variables
on $[0,1]$. Let $\varepsilon_n$, $n\in\mathbb{N}$ be another such sequence, independent of 
$(U_n)_{n\in\mathbb{N}}$. By enlarging the probability space, if necessary, we can always construct
such random variables and we may even assume that $(U_n,\varepsilon_n)$, $n\in\mathbb{N}$
are independent of $(X_0,(Y_t)_{t\in\mathbb{Z}})$. 

We assume that $\mathcal{X}$
is uncountable, the case of countable $\mathcal{X}$ being analogous, but simpler. 
As $\mathcal{X}$ is Borel-isomorphic to $\mathbb{R}$, see page 159 of \cite{dm}, we may and will
assume that, actually, $\mathcal{X}=\mathbb{R}$ (we omit the details). 

The main idea in the arguments below is to separate the ``independent component'' $\alpha(n)\nu_n(\cdot)$
from the rest of the kernel $Q(y,x,\cdot)-\alpha(n)\nu_n(\cdot)$ for $y\in A_n$ and $x\in C(R(n))$.
This independent component will ensure the existence of the constant mappings in \eqref{patty}.

Recall the sets $A_n$, $n\in\mathbb{N}$ from Assumption \ref{lyapunov}.
Let $B_n:=A_n\setminus A_{n-1}$, $n\in\mathbb{N}$, with the convention $A_{-1}:=\emptyset$.
For each $n\in\mathbb{N}$, $y\in B_n$, let $j_n(y,r):=\nu_{n}((-\infty,r])$, $r\in\mathbb{R}$ (the cumulative distribution
function of $\nu_n$) and
define its ($\mathfrak{A}\otimes\mathcal{B}(\mathbb{R})$-measurable) pseudoinverse by 
$j^-_n(y,z):=\inf\{r\in\mathbb{Q}:\, j(y,r)\geq z\}$, $z\in\mathbb{R}$. Here $\mathcal{B}(\mathbb{R})$
refers to the Borel-field of $\mathbb{R}$.
Similarly, for $y\in B_n$ and $x\in C(R(n))$, let
$$
q(y,x,r):=\frac{Q(y,x,(-\infty,r])-\alpha(n)j_n(y,r)}{1-\alpha(n)},\ r\in\mathbb{R},
$$
the cumulative distribution function of the normalization of $Q(y,x,\cdot)-\alpha(n)\nu_n(\cdot)$.
For $x\notin C(R(n))$, set simply
$$
q(y,x,r):=Q(y,x,(-\infty,r]),\ r\in\mathbb{R}.
$$
For each $x\in\mathcal{X}$, define
$$
q^-(y,x,z):=\inf\{r\in\mathbb{Q}:\, q(y,x,r)\geq z\},\ z\in\mathbb{R}.
$$

Define, for $n\in\mathbb{N}$, $y\in B_n$,
\begin{eqnarray*}
T_t(y,x,\omega)  &:=& q^-(y,x,\varepsilon_t),\mbox{ if }U_t(\omega)>\alpha(n)
\mbox{ or }U_t(\omega)\leq\alpha(n)\mbox{ but }x\notin C(R(n)),\\
T_t(y,x,\omega)  &:=& j_n^-(y,\varepsilon_t),\mbox{ if }U_t(\omega)\leq\alpha(n)\mbox{ and }
x\in C(R(n)).
\end{eqnarray*}
Notice that $T_t(y,\cdot,\omega)\in \mathfrak{C}({R(\Vert y\Vert)})$ 
whenever $U_t(\omega)\leq\alpha(n)$, this implies \eqref{patty} with 
$J_t(y):=\{\omega:\, U_t(\omega)\leq\alpha(\Vert y\Vert)\}$.
The claimed independence of the sequence of sigma-algebras clearly holds.
It is easy to check \eqref{madrid}, too.
\end{proof}

\begin{remark}\label{kund}
{\rm Note that, in the above construction,
$(U_n,\varepsilon_n)_{n\in\mathbb{N}}$ was taken to be independent of
$(X_0,(Y_t)_{t\in\mathbb{Z}})$. This will be important later, in the proof of Theorem \ref{lln}.}
\end{remark}

We drop dependence of the mappings $T_t$ on $\omega$ in the notation from now on and
will simply write $T_t(y,x)$.
We continue our preparations for the proof of Theorem \ref{lln}.
Let $\mathcal{G}_t:=\sigma(\varepsilon_i,U_i,\ i\leq t)$ and
$\mathcal{G}^+_t:=\sigma(\varepsilon_i,U_i,\ i\geq t+1)$, $t\in\mathbb{N}$. Take an arbitrary
element $\tilde{x}\in\mathcal{X}$, this will remain fixed throughout this section.

Our approach to the ergodic theorem for $X$ does not rely on the Markovian structure, it proceeds 
rather through establishing a convenient mixing property.
The ensuing arguments will lead to Theorem \ref{lln} via the $L$-mixing
property of certain auxiliary Markov chains. It turns out that $L$-mixing is particularly
well-adapted to Markov chains, even when they are inhomogeneous (and for us
this is the crucial point). The main ideas of the arguments below go back to 
\cite{bm}, \cite{bw}, \cite{gmmtv} and \cite{msg}. In \cite{gmmtv} and \cite{msg}, Doeblin chains were treated. We 
need to extend those arguments substantially in the present, more complicated setting.

Let us fix $\mathbf{y}=(y_j)_{j\in\mathbb{Z}}\in\mathfrak{Y}$ till further notice such that,
for some $H\in\mathbb{N}$, $\Vert y_j\Vert\leq H$ holds for all $j\in\mathbb{Z}$.
Define $Z_0:=X_0$, $Z_{t+1}:=T_{t+1}({y}_t,Z_t)$, $t\in\mathbb{N}$.
Clearly, the process $Z$ heavily depends on the choice of $\mathbf{y}$. However, for a while
we do not signal this dependence for notational simplicity. Fix also $m\in\mathbb{N}$ till further notice.
Define $\tilde{Z}_m:=\tilde{x}$, $\tilde{Z}_{t+1}:=T_{t+1}({y}_t,\tilde{Z}_t)$, $t\geq m$. Notice that $\tilde{Z}_t$, $t\geq m$ are $\mathcal{G}^+_m$-measurable.

Our purpose will be to prove that, with a large probability,
$Z_{m+\tau}=\tilde{Z}_{m+\tau}$ for $\tau$ large enough. In other
words, a coupling between the processes $Z$ and $\tilde{Z}$ is
realized. Fix $\epsilon>0$ which will be specified later.
Let $\tau\geq 1$ be an arbitrary integer. 
Denote $\vartheta:=\lceil\, \epsilon\tau\, \rceil$.
Recall that $R(H)=8K(H)/\lambda(H)$. 
Define $D:=C(R(H)/2)=\{x\in\mathcal{X}:\, V(x)\leq R(H)/2\}$ and
$\overline{D}:=\{(x_1,x_2)\in\mathcal{X}^2:\, V(x_1)+V(x_2)\leq R(H)\}$.

Now let us notice that if $z\in\mathcal{X}\setminus D$, then 
for all $y\in A_H$,
\begin{eqnarray}\nonumber
[Q(y)(K(H)+V)](z)&\leq& (1-\lambda(H))V(z)+2K(H)\\
&\leq& (1-\lambda(H)/2)V(z).
\label{hsg}
\end{eqnarray}

Denote $\overline{Z}_t:=(Z_t,\tilde{Z}_t)$, $t\geq m$.
Define the $(\mathcal{G}_t)_{t\in\mathbb{N}}$-stopping times
$$
\sigma_0:=m,\ \sigma_{n+1}:=\min\{i>\sigma_n:\ \overline{Z}_i\in \overline{D}\}.
$$

\begin{lemma}\label{bobo}
We have $\sup_{k\in\mathbb{N}}E[V(Z_k)]\leq E[V(X_0)]+K(H)/\lambda(H)<\infty$. 
Furthermore, $\sup_{k\geq m}E[V(\tilde{Z}_k)]\leq V(\tilde{x})+K(H)/\lambda(H)$.
\end{lemma}
\begin{proof} Assumption \ref{lyapunov} easily implies that, for $k\geq 1$,
$$
E[V(Z_k)]\leq (1-\lambda(H))E[V(Z_{k-1})]+K(H).
$$
Assumption \ref{init} implies that $E[V(X_0)]=E[V(Z_0)]<\infty$ so, for every $k\in\mathbb{N}$,
\begin{equation*}
E[V(Z_k)]\leq E[V(X_0)]+\sum_{l=0}^{\infty} K(H)(1-\lambda(H))^l=E[V(X_0)]+\frac{K(H)}{\lambda(H)}.
\end{equation*}
Similarly, 
$$
E[V(\tilde{Z}_k)]\leq V(\tilde{x})+\sum_{l=0}^{\infty} K(H)(1-\lambda(H))^l= V(\tilde{x})+\frac{K(H)}{\lambda(H)}.
$$
\end{proof}

The counterpart of the above lemma for $X$ (driven by $Y$, which is stochastic) instead of $Z$ is the following.

\begin{lemma}\label{evx}
$$
\sup_{n\in\mathbb{N}}E[V(X_n)]<\infty.
$$
\end{lemma}
\begin{proof}
Note that $E[V(X_0)]<\infty$ by Assumption \ref{init}. So, for each $n\geq 1$,
\begin{eqnarray*}
E[V(X_n)]\leq \int_{\mathcal{X}}(1+V(z))\mu_n(dz) &\leq&\\ 
\int_{\mathcal{X}}(1+V(z))|\mu_n-\mu_0|(dz) +
\int_{\mathcal{X}}(1+V(z))\mu_0(dz) &=&\\
\rho_1(\mu_n,\mu_0)+ E[V(X_0)]+1.
\end{eqnarray*}
As $\rho_1(\mu_n,\mu_0)\to \rho_1(\mu_*,\mu_0)$ by Theorem \ref{limit}, the
statement follows.
\end{proof}

The results below serve to control the number of returns to $\overline{D}$ and the probability of coupling between the processes $Z$ and $\tilde{Z}$. Our estimation strategy in the proof of Theorem \ref{lln} will be the following. We will control $P(\tilde{Z}_{\tau+m}\neq Z_{\tau+m})$ for large $\tau$: either there were only few returns of the process $\overline{Z}$ to $\overline{D}$ (which happens with small probability) or there were
many returns but coupling did not occur (which also has small probability).
First let us present a lemma controlling the number of returns to 
$\overline{D}$.

\begin{lemma}\label{standage} There is 
$\bar{C}>0$ such that
$$
\sup_{n\geq 1} E\left[\exp(\varrho(H)(\sigma_{n+1}-\sigma_n))\big\vert
\mathcal{G}_{\sigma_n}\right]\leq \frac{\bar{C}}{\lambda^2(H)},
$$
and 
$$
E[\exp(\varrho(H)(\sigma_1-\sigma_0))]\leq \frac{\bar{C}}{\lambda^2(H)}
$$
where
$\varrho(H):=\ln(1+\lambda(H)/2)$. 
In particular, $\sigma_n<\infty$ a.s. for each $n\in\mathbb{N}$. Furthermore,
$\bar{C}$ does not depend on either $\mathbf{y}$, $m$ or $H$.
\end{lemma}
\begin{proof}  
We can estimate, for $k\geq 1$ and $n\geq 1$, 
\begin{eqnarray*}
P(\sigma_{n+1}-\sigma_n> k\vert\mathcal{G}_{\sigma_n})= P(\overline{Z}_{\sigma_n+k}\notin \overline{D},\ldots, \overline{Z}_{\sigma_n+1}
\notin\overline{D}\vert\mathcal{G}_{\sigma_n}) &\leq&\\
E\left[\left(\frac{V(Z_{\sigma_n+k})+V(\tilde{Z}_{\sigma_n+k})}{R(H)}\right)1_{\{\overline{Z}_{\sigma_n+k-1}\notin \overline{D}\}}\cdots
1_{\{\overline{Z}_{\sigma_n+1}\notin \overline{D}\}}\vert\mathcal{G}_{\sigma_n}\right] &=&\\
E\left[E\left[\left(\frac{V(Z_{\sigma_n+k})+V(\tilde{Z}_{\sigma_n+k})}{R(H)}\right)1_{\{\overline{Z}_{\sigma_n+k-1}\notin \overline{D}\}}
|\mathcal{G}_{\sigma_n+k-1}\right]1_{\{\overline{Z}_{\sigma_n+k-2}\notin \overline{D}\}}\right. &\cdots& \\
\left. \cdots 1_{\{\overline{Z}_{\sigma_n+1}\notin \overline{D}\}}\vert\mathcal{G}_{\sigma_n}\right]. & &
\end{eqnarray*}

Notice that, on $\{\overline{Z}_{\sigma_n+k-1}\notin \overline{D}\}$, either $Z_{\sigma_n+k-1}$
or $\tilde{Z}_{\sigma_n+k-1}$ falls outside $D$. Let us assume that $Z_{\sigma_n+k-1}$
does so, i.e.\ the estimation below is meant to take place on the set $\{Z_{\sigma_n+k-1}\notin D\}$. The other case can be treated analogously. Assumption \ref{lyapunov} and the 
observation \eqref{hsg} imply that
\begin{eqnarray*}
E\left[\left(\frac{V(Z_{\sigma_n+k})+V(\tilde{Z}_{\sigma_n+k})}{R(H)}\right)1_{\{\overline{Z}_{\sigma_n+k-1}\notin \overline{D}\}}|\mathcal{G}_{\sigma_n+k-1}\right] 
&\leq&\\
\frac{1}{R(H)}[(1-\lambda(H)/2) V(Z_{\sigma_n+k-1})-K(H)] &+&\\
\frac{1}{R(H)}[(1-\lambda(H))V(\tilde{Z}_{\sigma_n+k-1})+K(H)]
&\leq&\\ \frac{1-\lambda(H)/2}{R(H)} [V(Z_{\sigma_n+k-1})+V(\tilde{Z}_{\sigma_n+k-1})]. & &
\end{eqnarray*}
This argument can clearly be iterated and leads to
\begin{eqnarray*}
P(\sigma_{n+1}-\sigma_n> k\vert\mathcal{G}_{\sigma_n})&\leq&\\
\frac{(1-\lambda(H)/2)^{k-1}}{R(H)} E\left[V(Z_{\sigma_n+1})+V(\tilde{Z}_{\sigma_n+1})\Big\vert
\mathcal{G}_{\sigma_n}\right] &\leq&\\
\frac{(1-\lambda(H)/2)^{k-1}}{R(H)} \left[ (1-\lambda(H))\left[V(Z_{\sigma_n})+V(\tilde{Z}_{\sigma_n})\right]+2K(H)\right] &\leq &\\
\leq (1-\lambda(H)/2)^{k},  
\end{eqnarray*}
by Assumption \ref{lyapunov}, since $\overline{Z}_{\sigma_n}\in \overline{D}$. 
In the case $n=0$, we arrive at 
\begin{eqnarray*}
P(\sigma_{1}-\sigma_0> k)&\leq&\\
 E\left[(1-\lambda(H))(V(Z_m)+V(\tilde{x}))+2K(H)\right]
\frac{(1-\lambda(H)/2)^{k-1}}{R(H)} &\leq&\\
\left(E[V(X_0)]+\frac{1}{8}+V(\tilde{x})+\frac{\lambda(H)}{4}\right)
\left(1-\frac{\lambda(H)}{2}\right)^{k-1} & & 
\end{eqnarray*}
instead, in a similar way, by Lemma \ref{bobo}. 

Now we turn from probabilities to expectations.
Using $e^{\varrho(H)}\leq 2$,
we can estimate, for $n\geq 1$,
\begin{eqnarray*}
E\left[\exp\{\varrho(H)(\sigma_{n+1}-\sigma_n)\}\big\vert
\mathcal{G}_{\sigma_n}\right] &\leq&\\
\sum_{k=0}^{\infty} e^{\varrho(H)(k+1)}\left(1-\frac{\lambda(H)}{2}\right)^{k} &\leq&\\
2\sum_{k=0}^{\infty} \left(1-\frac{\lambda^2(H)}{4}\right)^{k} &=& \frac{8}{\lambda^2(H)}. 
\end{eqnarray*}

When $n=0$, we obtain 
\begin{eqnarray*}
E\left[\exp\{\varrho(H)(\sigma_{1}-\sigma_0)\}\right] &\leq&\\
\left(E[V(X_0)]+\frac{1}{8}+V(\tilde{x})+\frac{\lambda(H)}{4}\right)\left[
e^{\varrho(H)}+
\sum_{k=1}^{\infty} e^{\varrho(H)(k+1)}\left(1-\frac{\lambda(H)}{2}\right)^{k-1}\right] &\leq&\\
\frac{\bar{C}}{\lambda^2(H)}, & & 
\end{eqnarray*}
for some $\bar{C}\geq 8$. The statement follows.
\end{proof}

Now we make the choice 
$$
\epsilon:=\epsilon(H)=\varrho(H)/4(\ln(\bar{C})-2\ln(\lambda(H))).
$$

\begin{corollary}\label{dargay} If 
\begin{equation}\label{thor}
\tau\geq 1/\epsilon(H),
\end{equation}
then
\[
P(\sigma_{\vartheta}>m+\tau)\leq \exp(-\varrho(H)\tau/2).
\]
\end{corollary}
\begin{proof} Lemma \ref{standage} and the tower rule for conditional
expectations easily imply
$$
E[\exp(\varrho(H)\sigma_{\vartheta})]\leq \left(\frac{\bar{C}}{\lambda^2(H)}\right)^{\vartheta}e^{\varrho(H)m}.
$$
Hence, by the Markov inequality,
$$
P(\sigma_{\vartheta}>m+\tau)\leq \left(\frac{\bar{C}}{\lambda^2(H)}\right)^{\vartheta}\exp(-\varrho(H)\tau).
$$
The statement now follows by direct calculations. Indeed, this choice of $\epsilon(H)$ and $\tau\geq 1/\epsilon(H)$ imply
$$
(\ln(\bar{C})-2\ln(\lambda(H)))[\epsilon(H)\tau+1]\leq \frac{\tau}{2}\ln(1+\lambda(H)/2),
$$
which guarantees
\begin{equation*}
(\ln(\bar{C})-2\ln(\lambda(H))\lceil\epsilon(H)\tau\rceil-\tau\ln(1+\lambda(H)/2)\leq -\frac{\tau}{2}\ln(1+\lambda(H)/2).
\end{equation*}
\end{proof}

The next lemma controls the probability of coupling between $Z$ and $\tilde{Z}$.
\begin{lemma}\label{fonay} 
$$
P(Z_{m+\tau}\neq \tilde{Z}_{m+\tau},\ \sigma_{\vartheta}\leq m+\tau)\leq (1-\alpha(H))^{\vartheta-1}\leq e^{-(\vartheta-1)\alpha(H)}.
$$
\end{lemma}
\begin{proof} For typographical reasons, we will write $\sigma(n)$ instead of $\sigma_n$ in this proof.
Notice that if $\omega\in\Omega$ is such that 
$\sigma(k)(\omega)<m+\tau$ and $T_{\sigma(k)(\omega)+1}(y_{\sigma(k)(\omega)+1},
\cdot,\omega)\in
\mathfrak{C}(R(H))$ then $Z_{\sigma(k)(\omega)+1}(\omega)=\tilde{Z}_{\sigma(k)(\omega)
+1}(\omega)$ hence also $Z_{m+\tau}(\omega)=\tilde{Z}_{m+\tau}(\omega)$. Recall the proof of Lemma \ref{t} and
estimate
\begin{eqnarray*}
P(Z_{m+\tau}\neq \tilde{Z}_{m+\tau},\ \sigma({\vartheta})\leq m+\tau) &\leq&\\
P(U_{\sigma(1)+1}>\alpha(H),\ldots, U_{\sigma({\vartheta-1})+1}>\alpha(H)) &=&\\
E[E[1_{\{U_{\sigma({\vartheta-1})+1}>\alpha(H)\}}\vert\mathcal{G}_{\sigma({\vartheta-1})}] 
1_{\{U_{\sigma(1)+1}>\alpha(H)\}}\cdots 1_{\{U_{\sigma(\vartheta-2)+1}>\alpha(H)\}}]. & &
\end{eqnarray*}
As easily seen, 
\begin{eqnarray*}
E[1_{\{U_{\sigma({\vartheta-1})+1}>\alpha(H)\}}
\vert\mathcal{G}_{\sigma({\vartheta-1})}] = (1-\alpha(H)).
\end{eqnarray*}
Iterating the above argument, we arrive at the statement of this lemma using
$1-x\leq e^{-x}$, $x\geq 0$.
\end{proof}

\begin{lemma}\label{mommo} 
Let $\phi\in\Phi(V^{\delta})$
for some $0<\delta\leq 1/2$. 
Then the process $\phi(Z_t)$, $t\in\mathbb{N}$ is $L$-mixing of order $p$ 
with respect to $(\mathcal{G}_t,\mathcal{G}^+_t)$, $t\in\mathbb{N}$, for all $1\leq p<1/\delta$. 
Furthermore, $\Gamma_{p}(\phi(Z))$, $M_{p}(\phi(Z))$ have upper bounds that do not depend on $\mathbf{y}$, only on $H$. 
\end{lemma}

In the sequel we will use, without further notice, the following elementary inequalities for $x,y\geq 0$:
$$
(x+y)^r\leq 2^{r-1}(x^r+y^r)\mbox{ if }r\geq 1;\ (x+y)^r\leq x^r+y^r\mbox{ if }0<r<1.
$$

\begin{proof}[Proof of Lemma \ref{mommo}.] 
Clearly, 
\begin{equation*}
M_{1/\delta}(\phi(Z))\leq \tilde{C}\left[1+\left(E[V(X_0)]+\frac{K(H)}{\lambda(H)}\right)^{\delta}\right],
\end{equation*}
by Lemma \ref{bobo}. Also,
$$
M_p(\phi(Z))\leq M_{1/\delta}(\phi(Z)),
$$
for all $1\leq p<1/\delta$.

Now we turn to establishing a bound for $\Gamma_p(\phi(Z))$.
Since $\tilde{Z}_m$ is deterministic, $\tilde{Z}_{m+\tau}$ is $\mathcal{G}_m^+$-measurable. Lemma \ref{fyffes} implies that, 
for $\tau\geq 1$,
\begin{eqnarray}\nonumber
E^{1/p}[|\phi(Z_{m+\tau})-E[\phi(Z_{m+\tau})\vert\mathcal{G}_{m}^+]|^p] &\leq&\\
\nonumber 2E^{1/p}[|\phi(Z_{m+\tau})-\phi(\tilde{Z}_{m+\tau})|^p] &\leq&\\ 
\nonumber 2E^{1/p}[(|\phi(Z_{m+\tau})|+|\phi(\tilde{Z}_{m+\tau})|)^p 1_{\{Z_{m+\tau}\neq \tilde{Z}_{m+\tau}\}}] &\leq&\\
2 E^{\delta}[(|\phi(Z_{m+\tau})|+|\phi(\tilde{Z}_{m+\tau})|)^{1/\delta}]P^{\frac{1-p\delta}{p}}(Z_{m+\tau}\neq \tilde{Z}_{m+\tau}), & &
\label{orlando}
\end{eqnarray}
using H\"older's inequality with the exponents $1/(p\delta)$ and $1/(1-p\delta)$. By Lemma \ref{bobo}, 
\begin{eqnarray}\nonumber
E^{\delta}[(|\phi(Z_{m+\tau})|+|\phi(\tilde{Z}_{m+\tau})|)^{1/\delta}] &\leq&\\
\nonumber \tilde{C}\left[1+\left( E[V(X_0)]+\frac{K(H)}{\lambda(H)}\right)^{\delta}\right] &+&\\
\tilde{C}\left[1+\left( V(\tilde{x})+\frac{K(H)}{\lambda(H)}\right)^{\delta}\right] 
&\leq&
\check{C}\left[\frac{K(H)}{\lambda(H)}\right]^{\delta},\label{nape}
\end{eqnarray}
for some suitable $\check{C}>0$. Since
$$
P(Z_{m+\tau}\neq \tilde{Z}_{m+\tau})\leq P(Z_{m+\tau}\neq \tilde{Z}_{m+\tau},\ \sigma_{\vartheta}\leq m+\tau)
+ P(\sigma_{\vartheta}> m+\tau),
$$
we obtain from Lemma \ref{fonay} and Corollary \ref{dargay} that for $\tau$ satisfying \eqref{thor},
\begin{eqnarray*}& &
\gamma_p(\phi(Z),\tau) \\ 
&\leq& 2\check{C}\left(\frac{K(H)}{\lambda(H)}\right)^{\delta} \left[
\exp\left(-\alpha(H)[\epsilon(H)\tau-1](1-p\delta)/p\right)+ 
\exp\left(-\frac{\varrho(H)\tau}{2}(1-p\delta)/p\right)\right], 
\end{eqnarray*}
noting that the estimates of Lemma \ref{fonay} and Corollary \ref{dargay} do not
depend on the choice of $m$. For each integer
$$
1\leq \tau<1/\epsilon(H),
$$ 
we will apply the trivial estimate 
$$
\gamma_p(\phi(Z),\tau)\leq 2M_p(\phi(Z))\leq 2M_{1/\delta}(\phi(Z))\leq 2\check{C}\left[\frac{K(H)}{\lambda(H)}\right]^{\delta},
$$
recall \eqref{nape}. Hence
\begin{eqnarray}\nonumber
\Gamma_p(\phi(Z))\leq 2\check{C}\frac{1}{\epsilon(H)}\left(\frac{K(H)}{\lambda(H)}\right)^{\delta} &+&\\
\nonumber  2\check{C}\sum_{\tau\geq 1/\epsilon(H)}
\left[
\exp\left(-\alpha(H)[\epsilon(H)\tau-1](1-p\delta)/p\right)+ 
\exp\left(-\frac{\varrho(H)\tau}{2}(1-p\delta)/p\right)\right]\left(\frac{K(H)}{\lambda(H)}\right)^{\delta} 
&\leq&\\ 
\nonumber c'\left[\frac{1}{\epsilon(H)}+\frac{\exp\left({\alpha(H)}(1-p\delta)/p\right)}
{1-\exp\left(-\alpha(H)\epsilon(H)(1-p\delta)/p\right)}+
\frac{1}{1-\exp\left( -\frac{\varrho(H)(1-p\delta)}{2p}\right)}\right]\left(\frac{K(H)}{\lambda(H)}\right)^{\delta} &\leq&\\
\nonumber c''\left[\frac{1}{\alpha(H)\epsilon(H)}+
\frac{1}{\lambda(H)}\right]
\left(\frac{K(H)}{\lambda(H)}\right)^{\delta} &\leq&\\
c'''\frac{|\ln(\lambda(H))|}{\alpha(H)\lambda(H)}\left(\frac{K(H)}{\lambda(H)}\right)^{\delta} 
\label{oo} & &
\end{eqnarray}
with some $c',c'',c'''>0$, using elementary properties
of the functions $x\to 1/(1-e^{-x})$ and $x\to \ln(1+x)$. The $L$-mixing property of order $p$ follows. (Note, however, that $c'''$
depends on $p$, $\delta$ as well as on $E[V(X_0)]$.)
\end{proof}

\begin{proof}[Proof of Theorem \ref{lln}.] Now we start signalling the dependence of $Z$ on
$\mathbf{y}$ and hence write $Z_t^{\mathbf{y}}$, $t\in\mathbb{N}$. 
Note that
the law of $Z_t^{\mathbf{Y}}$, $t\in\mathbb{N}$ equals that of $X_t$, $t\in\mathbb{N}$,
by construction of $Z$ and by Remark \ref{kund}.

For $t\in\mathbb{N}$ and $\mathbf{y}\in\mathfrak{Y}$, define $\psi_{t}(\mathbf{y}):=E[\phi(Z^{\mathbf{y}}_{t})]$ and
$W_t(\mathbf{y}):=\phi(Z^{\mathbf{y}}_t)-\psi_{t}(\mathbf{y})$.
Clearly, $W_{t}(\mathbf{y})$ is a zero-mean process.
  
Fix $p\geq 2$.  Fix $N\in\mathbb{N}$ for the moment. In the particular case where $\mathbf{y}$ satisfies $|y_j|\leq g(N)$, $j\in\mathbb{N}$,
the process  $W_t(\mathbf{y})$,
$t\in\mathbb{N}$ is $L$-mixing by Lemma \ref{mommo} and Remark \ref{utu}. Hence Lemma \ref{inek} implies
\begin{eqnarray*}\nonumber 
E^{1/p}\left[\left|\frac{W_1(\mathbf{y})+\ldots+W_N(\mathbf{y})}{N}\right|^p\right] &\leq&\\
\frac{C_p M_p^{1/2}(W(\mathbf{y}))\Gamma_p^{1/2}(W(\mathbf{y}))}{N^{1/2}}&\leq&\\
\frac{C_p M_{1/\delta}^{1/2}(W(\mathbf{y}))\Gamma_p^{1/2}(W(\mathbf{y}))}{N^{1/2}}&\leq&\\
\frac{2C_p \sqrt{\check{C}}   [K(g(N))/\lambda(g(N))]^{\delta/2} \sqrt{c'''}[K(g(N))/\lambda(g(N))]^{\delta/2} 
\pi^{1/2}(N)}{N^{1/2}}, & &\label{ae}
\end{eqnarray*}
by \eqref{nape} and \eqref{oo}; recall also Remark \ref{utu}. Fix $\tilde{y}\in A_0$ and define 
$$
\tilde{Y}_j:={Y}_j,\mbox{ if }{Y}_j\in A_{g(N)},\
\tilde{Y}_j:=\tilde{y},\mbox{ if }{Y}_j\notin A_{g(N)}.
$$
Let $\tilde{\mathbf{Y}}=(\tilde{Y}_{j})_{j\in\mathbb{Z}}\in
\mathfrak{Y}$.
Note that, by $\phi\in\Phi(V^{\delta})$,
$$
E^{\delta}[|W_j(\mathbf{Y})|^{1/\delta}]\leq 2\tilde{C}(1+E^{\delta}[V(X_j)]),\ j\geq 1.
$$
Estimate, using H\"older's inequality with exponents $1/(\delta p)$, $1/(1-\delta p)$,
\begin{eqnarray}
%\nonumber
%E^{1/p}\left[\left|\frac{(\phi(X_1)-E[\phi(X_1)])+\ldots+(\phi(X_N)-E[\phi(X_N)])}{N}\right|^p\right] &=&\\
\nonumber E^{1/p}\left[\left|\frac{W_1(\mathbf{Y})+\ldots+W_N(\mathbf{Y})}{N}\right|^p\right] &\leq&\\  
\nonumber E^{1/p}\left[\left|\frac{W_1(\tilde{\mathbf{Y}})+\ldots+W_N(\tilde{\mathbf{Y}})}{N}\right|^p\right] 
 &+&\\
\nonumber M_{1/\delta}(W(\mathbf{Y}))P^{\frac{1-p\delta}{p}}((\tilde{Y}_1,\ldots,\tilde{Y}_N)\neq (Y_1,\ldots,Y_N)) &\leq&\\ 
\nonumber\frac{C'[K(g(N))/\lambda(g(N))]^{\delta} 
\pi^{1/2}(N)}{N^{1/2}}+C'\left(1+\sup_{n\in\mathbb{N}} E[V(X_n)]\right)^{\delta}\ell^{\frac{1-p\delta}{p}}(N) 
&\leq &\\
\label{ollala}\frac{C''[K(g(N))/\lambda(g(N))]^{\delta} 
\pi^{1/2}(N)}{N^{1/2}}+C''\ell^{\frac{1-p\delta}{p}}(N),
\end{eqnarray}
with some constants $C',C''>0$, by Lemma \ref{evx}. Here we have also used the fact that if $(\tilde{Y}_1,\ldots,\tilde{Y}_N)=(Y_1,\ldots,Y_N)$
then also $W_{j}(\mathbf{Y})=W_{j}(\tilde{\mathbf{Y}})$, $j=1,\ldots,N$.
 
Recall the notation $\hat{\mathbf{Y}}_{n}:=(Y_{j+n})_{j\in\mathbb{Z}}$ and the definition of 
$\mu_{n}(\mathbf{y})$ from \eqref{dafie}. Recall also the functional $\Psi_{\phi}$ from
\eqref{defie}.
Now we can estimate
\begin{eqnarray}\nonumber
& & \left|\int_{\mathcal{X}} \phi(z)\mu_*(dz)-\frac{\sum_{j=1}^{N}\phi(Z_j^{\mathbf{Y}})}{N}\right|\\ 
\nonumber &\leq & \left|\int_{\mathcal{X}} \phi(z)\mu_*(dz)-\frac{\sum_{j=1}^{N}\Psi_{\phi}(\hat{\mathbf{Y}}_{j})}{N}\right|\\
\nonumber &+& \left|\frac{\sum_{j=1}^{N}\Psi_{\phi}(\hat{\mathbf{Y}}_{j})}{N}-\frac{\sum_{j=1}^{N}\psi_{j}(\hat{\mathbf{Y}}_{j})}{N}\right|\\
&+& \left|\frac{\sum_{j=1}^{N}\psi_{j}(\hat{\mathbf{Y}}_{j})}{N}-\frac{\sum_{j=1}^{N}\phi(Z_j^{\mathbf{Y}})}{N}\right|.
 \label{matee}
\end{eqnarray}
Notice that the law of 
$Z_{j}^{\mathbf{y}}$ equals $\mu_{j}(\hat{\mathbf{y}}_{j})$ so 
the third term on the right-hand side of \eqref{matee} equals
$$\left|\frac{W_1(\mathbf{Y})+\ldots+W_N(\mathbf{Y})}{N}\right|$${}
hence it converges to $0$ in probability by \eqref{ollala}.

By stationarity, we get
$$
E|\psi_{j}(\hat{\mathbf{Y}}_{j})-\Psi_{\phi}(\hat{\mathbf{Y}}_{j})|=E|\psi_{j}(\mathbf{Y})-\Psi_{\phi}(\mathbf{Y})|\to 0
$$
as $j\to\infty$, see the proof of Theorem \ref{limit}, so the second term also tends to $0$ in probability. 

%The first term tends to zero since $E[\phi(X_t)]$ converges to 
%$\int_{\mathcal{X}} \phi(x)\, \mu_*(dx)$ by Theorem \ref{limit}. 

Finally, Birkhoff's theorem and the ergodicity of the process $Y$ imply that
\begin{eqnarray*}
\frac{\sum_{j=1}^{N}\Psi_{\phi}(\hat{\mathbf{Y}}_{j})}{N}\to \int_{\mathcal{X}}\phi(z)\mu_{*}(dz),\ N\to\infty,
\end{eqnarray*}
almost surely, hence also in probability, noting Remark \ref{wekings}. This shows that the first
term on the right-hand side of \eqref{matee} also vanishes. To sum up, 
$$
\left|\int_{\mathcal{X}} \phi(z)\mu_*(dz)-\frac{\sum_{j=1}^{N}\phi(X_j)}{N}\right|\to 0
$$
in probability, recalling that the laws of $Z^{\mathbf{Y}}_{n}$, $n\in\mathbb{N}$ and $X_{n}$, $n\in\mathbb{N}$
coincide.

To show convergence in $L^{p}$, it suffices to check the uniform integrability of the family of 
random variables $V^{\delta p}(X_{n})$, $n\in\mathbb{N}$
since $\phi\in\Phi(V^{\delta})$. This follows from $p<1/\delta$ and from Lemma \ref{evx}. 
The theorem has been shown for $p\geq 2$ but this implies the result for $1\leq p<2$, too. 
\end{proof}

\begin{remark}\label{speed}
{\rm In \eqref{matee} in the proof of Theorem \ref{lln} we can find estimates for the $L^{p}$ convergence rate
for every term except for
$$
\mathfrak{e}(N):=\left|\frac{\sum_{j=1}^{N}\Psi_{\phi}(\hat{\mathbf{Y}}_{j})}{N}- 
\int_{\mathcal{X}}\phi(z){}
\mu_{*}(dz)
\right|.
$$
Making suitably strong (mixing) assumptions about the process $Y$, however, this term can also be estimated.
In the ideal case, $E^{1/p}[\mathfrak{e}_{N}^{p}]$ is of the order $1/\sqrt{N}$.} 
\end{remark}

\section{Proofs of ergodicity II}\label{mehi}

\begin{proof}[Proof of Theorem \ref{llnn}.] This follows very closely the proof of Theorem \ref{lln}, we only
point out the differences. Denote by $S$ an upper bound for $|\phi|$.
Take an arbitrary $p\geq 2$. We may use the H\"older inequality with exponents $1$ and $\infty$ in 
the estimates \eqref{orlando}. This leads to 
$$
\Gamma_p(\phi(Z))\leq c'''\frac{|\ln(\lambda(H))|}{\alpha(H)\lambda(H)},
$$
using the argument of \eqref{oo}. Then the proof of convergence in probability can be completed as above. Note that, instead of  
$$
M_{1/\delta}(W(\mathbf{Y}))P^{\frac{1-p\delta}{p}}((\tilde{Y_1},\ldots,\tilde{Y}_N)\neq (Y_1,\ldots,Y_N))
$$
we may write
$$
SP((\tilde{Y_1},\ldots,\tilde{Y}_N)\neq (Y_1,\ldots,Y_N))\leq S\ell(N)
$$
in \eqref{ollala}. As $\phi$ is bounded, $L^{p}$ convergence for all $p\geq 1$ also follows. 
%Finally, we arrive at
%\begin{eqnarray}\nonumber
%E^{1/p}\left|\frac{\phi(X_1)+\ldots+\phi(X_N)}{N}-\int_{\mathcal{X}} \phi(z)\mu_*(dz)\right|^p &\leq&\\ 
%C\left[\sqrt{\frac{\pi(N)}{N}}+\frac{\sum_{j=1}^N [r_3(j)+r_4(j)]}{N}\right], & &\label{bonnz}
%\end{eqnarray}
%for some $C=C(p)>0$, noting that, since $\ell(N)\leq [\sum_{j=1}^N r_4(j)]/N$, 
%the term containing $\ell(N)$ is subsumed
%in the convergence rate in \eqref{bonnz}. 
\end{proof}

\begin{example}{\rm Let $X_t$, 
$t\in\mathbb{N}$ be a 
$\mathcal{X}$-valued
\emph{Markov chain} with $X_0=x_0$, where $\mathcal{X}$ is a Polish space with
Borel field $\mathfrak{B}$. Denoting the transition kernel of $X$ by
$Q(x,A)$, $x\in\mathcal{X}$, $A\in\mathfrak{B}$, we impose two standard
assumptions (see \cite{mt,hm}) for geometric ergodicity:
\begin{equation*}
[QV](x)\leq (1-\lambda)V(x)+K,\ x\in\mathcal{X},
\end{equation*}
for some measurable function $V:\mathcal{X}\to\mathbb{R}_{+}$, $0<\lambda\leq 1$,
$K>0$ and 
\begin{equation*}
\inf_{x\in C} Q(x,A)\geq \alpha \nu(A),\ A\in\mathfrak{B},
\end{equation*}
for some probability $\nu$, constant $\alpha>0$ and 
$$
C:=\{x\in\mathcal{X}:\, V(x)\leq 4K/\lambda\}.
$$
Under these assumptions, the process $X$ fits our framework above
(choosing $\mathcal{Y}$ to be a singleton) and 
the arguments of Lemma \ref{mommo} show that, for $0< \delta\leq 1/2$ and for 
any $\phi\in\Phi(V^{\delta})$, the process $\phi(X_t)$ is $L$-mixing of order $p$ for each
$1\leq p<1/\delta$. Furthermore, $$
M_p(\phi(X))+\Gamma_p(\phi(X))\leq \bar{c}[1+V^{\delta}(x_0)]
$$
for some $\bar{c}=\bar{c}(p)>0$. When $\phi$ is bounded, the same holds for each
$p\geq 1$. 

Although this result forms a very particular case of our framework, it is new and of considerable interest: on one hand, it establishes a useful 
mixing property for functionals of a wide class of Markov processes; on the other
hand, it underlines the versatility of the concept of $L$-mixing by providing one more relevant class
of examples satisfying this notion of mixing.}
\end{example}

\noindent\textbf{Acknowledgments.} Both authors enjoyed the support of 
the NKFIH (National Research, Development and Innovation Office, Hungary) 
grant KH 126505. The first author 
was also supported by the NKFIH grant PD 121107; the second
author by the ``Lend\"ulet'' grant LP 2015-6 of the
Hungarian Academy of Sciences and
by The Alan Turing Institute, London under the EPSRC grant EP/N510129/1. We thank Attila Lovas for
pointing out two mistakes and for suggesting improvements. The paper also benefitted from comments by
Nicolas Brosse, \'Eric Moulines, Sotirios Sabanis and Ramon van Handel.


\begin{thebibliography}{99}

\bibitem{bm} R. Bhattacharya and M. Majumdar.
\newblock On a theorem of Dubins and Freedman.
\newblock \emph{J. Theor. Probab.}, 12:1067--1087, 1999.

\bibitem{bwbook}
R.~N. Bhattacharya and E.~C. Waymire.
\newblock \emph{Stochastic Processes with Applications}.
\newblock Wiley \& Sons, New York, 1990.

\bibitem{bw} R. Bhattacharya and E. C. Waymire.
\newblock An approach to the existence
of unique invariant probabilities for Markov processes.
\newblock In: \emph{Limit theorems
in probability and statistics, J\'anos Bolyai Math. Soc., I}, 181--200, 2002.

\bibitem{borovkov}
A. A. Borovkov.
\newblock\emph{Egodicity and stability of stochastic processes.}
\newblock Wiley \& Sons, New York, 1998.

\bibitem{c1} R. Cogburn.
\newblock The ergodic theory of Markov chains in random environments.
\newblock \emph{Z. Wahrsch. Verw. Gebiete}, 66:109--128, 1984.

\bibitem{c2} R. Cogburn.
\newblock On direct convergence and periodicity for transitions probabilities 
of Markov chains in random environments. 
\newblock \emph{Ann. Probab.}, 18:642--654, 1990. 

\bibitem{cr}
F. Comte and \'E. Renault.
\newblock Long memory in continuous-time stochastic volatility models.
\newblock {\em Math. Finance}, 8:291--323, 1998.

\bibitem{cont} R. Cont.
\newblock Empirical properties of asset returns: stylized facts and statistical issues.
\newblock \emph{Quantitative Finance}, 1:223--236, 2001.

\bibitem{dm}
C. Dellacherie and P.-A. Meyer.
\newblock\emph{Probability and potential.}
\newblock North-Holland, Amsterdam, 1979.

\bibitem{gjr}
J. Gatheral, T. Jaisson and M. Rosenbaum. \newblock Volatility is rough.
\newblock\emph{Quantitative Finance}, 18:933--949, 2018. 

\bibitem{laci1} L. Gerencs\'er.
\newblock On a class of mixing processes.
\newblock \emph{Stochastics},  26:165--191, 1989.

%\bibitem{laci} L. Gerencs\'er.
%\newblock Rate of convergence of recursive estimators.
%\newblock \emph{SIAM J. Control Optim.}, 30:1200--1227, 1992.

\bibitem{gmmtv} L. Gerencs{\'e}r, G. Moln\'ar-S\'aska, Gy. Michaletzky, G. Tusn\'ady and Zs. V\'ag\'o.
\newblock New methods for the statistical analysis of Hidden Markov models.
\newblock In: \emph{Proceedings of the 41st IEEE Conference on Decision and Control, 2002, Las Vegas, USA}
\newblock 2272--2277, IEEE Press, New York, 2002.
 
\bibitem{hm} M. Hairer and J. Mattingly. 
\newblock Yet another look at Harris' ergodic theorem for Markov chains.
\newblock \emph{In: Seminar on stochastic analysis, random fields and applications VI (eds.
R. Dalang, M. Dozzi and F. Russo F.)}, Progress in Probability, vol. 63, 109--117, 2011.

\bibitem{hms} M. Hairer, J. Mattingly and M. Scheutzow.
\newblock Asymptotic coupling and a general form of Harris' theorem with applications to 
stochastic delay equations.
\newblock \emph{Probab. Theory Related Fields}, 149:223--259, 2011.

\bibitem{kifer1} Y. Kifer. \newblock Perron-Frobenius theorem, large deviations, and random perturbations 
in random environments. \newblock\emph{Math. Zeitschrift}, 222:677--698, 1996.

\bibitem{kifer} Y. Kifer. \newblock Limit theorems for random transformations and
processes in random environments. \newblock\emph{Trans. American Math. Soc.}, 350:1481--1518, 1998.

\bibitem{mt} S. P. Meyn and R. L. Tweedie. 
\newblock \emph{Markov chains and stochastic stability.}
\newblock Springer-Verlag, 1993.

\bibitem{msg} G. Moln\'ar-S\'aska.
\newblock \emph{Statistical analysis of hidden Markov models.}
\newblock PhD thesis, E\"otv\"os Lor\'and University, 2005.

\bibitem{o} S. Orey.
\newblock Markov chains with stochastically stationary transition probabilities.
\newblock \emph{Ann. Probab.}, 19:907--928, 1991.


\bibitem{sep} T. Sepp\"al\"ainen. \newblock Large deviations for Markov chains with random transitions. \newblock\emph{Ann. Probab.}, \newblock 22:713--748, 1994.

\bibitem{stenflo} \"O. Stenflo. Markov chains in random environments and
random iterated function systems. \emph{Trans. American Math. Soc.}, vol. 353, 3547--3562, 2001.


\bibitem{tsirelson} B. S. Tsirelson.
\newblock A geometric approach to maximum likelihood estimation for
infinite-dimensional Gaussian location. II.
\newblock \emph{Theory Probab. Appl.}, 30:820--828, 1985.

\bibitem{vitale} R. A. Vitale.
\newblock The Wills functional and Gaussian processes.
\newblock \emph{Ann. Probab.}, 24:2172--2178, 1996.

\end{thebibliography}
\end{document}